\documentclass[dvipsnames
]{amsart}
\usepackage[utf8]{inputenc}
\usepackage{graphicx,stmaryrd,enumerate,mathtools}
\usepackage{amssymb}
\usepackage{amsmath,tensor,mathbbol}
\usepackage{float}
\usepackage{xcolor,tikz-cd}
\usepackage[backend = biber,citestyle=ieee-alphabetic, bibstyle=alphabetic ,maxnames=4,minnames=3,maxbibnames=99]{biblatex}
\usepackage{mathscinet}
\usepackage[justification=raggedleft]{caption}

\allowdisplaybreaks

\usepackage{hyperref}
\usepackage{cleveref}
\usepackage{comment}
\usepackage[all]{xy}
\DeclareFontFamily{U}{mathx}{}
\DeclareFontShape{U}{mathx}{m}{n}{<-> mathx10}{}
\DeclareSymbolFont{mathx}{U}{mathx}{m}{n}
\DeclareMathAccent{\widehat}{0}{mathx}{"70}
\DeclareMathAccent{\widecheck}{0}{mathx}{"71}

\numberwithin{equation}{section}

\newtheorem{thm}{Theorem}[section] 
\newtheorem{prop}[thm]{Proposition}
\newtheorem{lem}[thm]{Lemma}
\newtheorem{cor}[thm]{Corollary}

\theoremstyle{definition}
\newtheorem{defn}[thm]{Definition}
\newtheorem{notation}[thm]{Notation}
\newtheorem{ex}[thm]{Example}

\theoremstyle{remark}
\newtheorem{rem}[thm]{Remark}

\newcommand{\kk}{\mathbb{Z}}
\newcommand{\at}{{\mathtt{a}}}
\newcommand{\atb}{{\mathtt{b}}}
\newcommand{\pcoset}{{p}}
\newcommand{\qcoset}{{q}}
\newcommand{\rcoset}{{r}}

\def\hat{\widehat}

\newcommand{\pa}{\partial}

\DeclareMathOperator{\Hom}{Hom}

\newcommand{\inv}{^{-1}}

\newcommand{\cham}{\mathsf{Cham}}
\newcommand{\cone}{\mathsf{Cone}}

\newcommand{\mi}{\underline}
\newcommand{\ma}{\overline}
\newcommand{\expr}{\leftrightharpoons}

\newcommand{\Dem}{\mathcal{SD}}

\newcommand{\SC}{\mathcal{SC}}

\DeclareMathOperator{\Com}{\mathsf{Com}}
\DeclareMathOperator{\leftdes}{LD}
\DeclareMathOperator{\rightdes}{RD}
\DeclareMathOperator{\leftred}{LR}
\DeclareMathOperator{\rightred}{RR}
\DeclareMathOperator{\core}{core}

\newcommand{\NCA}{D}

\newcommand{\coreJ}[1]{\SC^{\core}_{#1}}

\newcommand{\con}{\bar}
\newcommand{\scon}{\textcolor{Orange}{\con{s}}}
\newcommand{\tcon}{\textcolor{Magenta}{\con{t}}}
\newcommand{\ucon}{\con{u}}
\newcommand{\Lcon}{\con{L}}

\newcommand{\ka}[1]{\!_{\langle #1\rangle}a}
\newcommand{\kb}[1]{\!_{\langle#1\rangle}b}

\definecolor{Mred}{RGB}{236,6,36}
\definecolor{Mgreen}{RGB}{0,158,88}
\definecolor{Mblue}{RGB}{71,71,151}

\newcommand{\sblue}{\textcolor{RoyalBlue}{s}}
\newcommand{\teal}{\textcolor{Green}{t}}

\newcommand{\red}{\textcolor{Mred}{r}}
\newcommand{\gr}{\textcolor{Mgreen}{g}}
\newcommand{\bl}{\textcolor{Mblue}{b}}

\newcommand{\rev}[1]{{\color{black} {#1}}}
\newcommand{\rrev}[1]{{\color{black} {#1}}}

\title{An atomic Coxeter presentation}
\author[Hankyung Ko]{Hankyung Ko}
\address{Department of Mathematics, Uppsala University,
Box. 480,
SE-75106, Uppsala, Sweden}
\email{hankyung.ko@math.uu.se}

\begin{document}

\begin{abstract}
We study parabolic double cosets in a Coxeter system  by decomposing them into atom(ic coset)s, a generalization of simple reflections introduced in a joint work with Elias, Libedinsky, Patimo.
We define and classify braid relations between compositions of atoms and prove a Matsumoto theorem. 
Together with a quadratic relation, our braid relations give a presentation of nilCoxeter algebroids similar to Demazure's presentation of nilCoxeter algebras. 
Our consideration of reduced compositions of atoms gives rise to a new combinatorial structure, which is equipped with a length function and a Bruhat order 
and is realized as Tits cone intersections in the sense of Iyama-Wemyss.
\end{abstract}
\maketitle
\section{Introduction}

\subsection{Coxeter complexes and core cosets}\label{ss.com}

Let $(W,S)$ be a Coxeter system.
For a subset $J\subset S$ we denote by $W_J$ the parabolic subgroup generated by $J$.
Recall (or see \cite{AbBr}) that the associated Coxeter complex $\Com(W,S)$ consists of cells (simplices)  indexed by proper parabolic left cosets, so that the cell $C_{wW_J}$ corresponding to $wW_J$, where $w\in W$ and $J\subsetneq S$, has dimension $|S\setminus J|-1$. In particular, the maximal cells in $\Com(W,S)$ are indexed by the one-element cosets $wW_\emptyset = \{w\}$.

The Coxeter group $W$ acts naturally on the left cosets in $(W,S)$, thus on the cells in $\Com(W,S)$. 
Let us fix a subset $I\subset S$ and consider the subcomplex $\Com(W,S,I)$ in $\Com(W,S)$ consisting of the cells fixed by the action of every $s\in I$. 
Then a cell $C_{wW_J}$ in $\Com(W,S)$ is contained in $\Com(W,S,I)$ if and only if the (left) action of $W_I$ on $wW_J$ is trivial, and the latter is if and only if $wW_J=W_IwW_J$.
Thus, given $I\subset S$, we can relabel the cells in $\Com(W,S,I)$ by such double cosets $W_IwW_J$.
The maximal cells in $\Com(W,S,I)$ are then indexed by what we call 
the \emph{core} $(W_I,W_J)$-cosets (where $J\subset S$ varies): 

\begin{defn}\label{introdef.core}
    A parabolic double coset $W_I w W_J$ is called a \emph{core coset} if $W_Iw = W_IwW_J = wW_J$ as sets.
\end{defn}

Here we illustrate the above discussions in an example.

\begin{ex}\label{ex.aA2}
Consider the affine symmetric group $(W,S)=(\widetilde{S}_3,\{\red,\gr,\bl\})$. Figure~\ref{rgb} shows $\Com(W,S)$. 
\begin{figure}
\centering
\begin{minipage}{.5\textwidth}
  \centering
  \includegraphics[width=.95\linewidth]{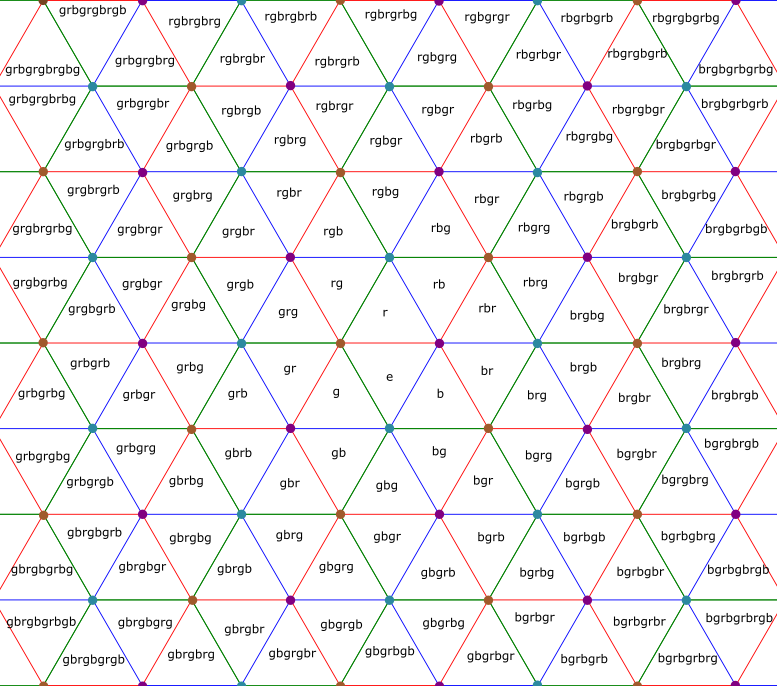}
  \captionof{figure}{$\Com(\widetilde{S}_3,\{\red,\gr,\bl\})$}
  \label{rgb}
\end{minipage}%
\begin{minipage}{.5\textwidth}
  \centering
  \includegraphics[width=.95\linewidth]{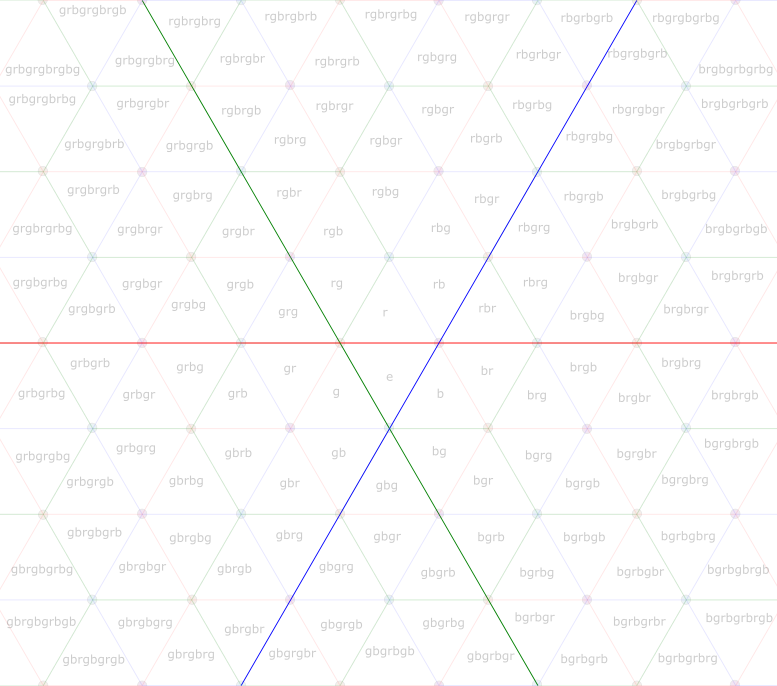}
  \captionof{figure}{Original $s$-hyperplanes}
  \label{rgbhyp}
\end{minipage}
\end{figure}
\begin{figure}
\centering
\begin{minipage}{.5\textwidth}
  \centering
  \includegraphics[width=.95\linewidth]{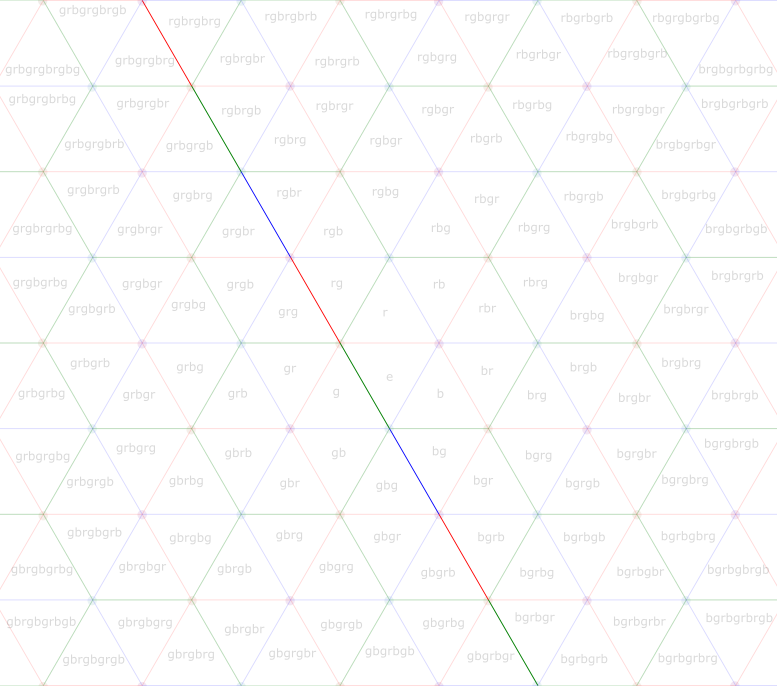}
  \captionof{figure}{\rrev{Maximal cells in} $\Com(\widetilde{S}_3,\{\red,\gr,\bl\},\{\gr\})$}
  \label{rgbg}
\end{minipage}%
\begin{minipage}{.5\textwidth}
  \centering
  \includegraphics[width=.95\linewidth]{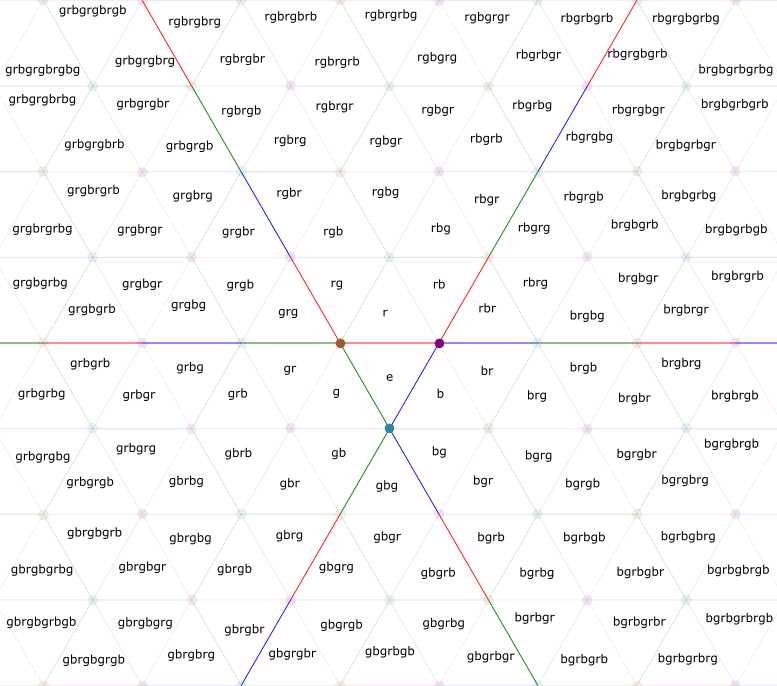}
  \captionof{figure}{All \rrev{cells corresponding to} core cosets}
  \label{rgbcore}
\end{minipage}
\end{figure}

\begin{itemize}
    \item The 2-dimensional (maximal) cells correspond to the elements $w\in W$, viewed as the trivial left cosets $wW_\emptyset$. One of the reduced expressions of each $w\in W$ indexes the cell in Figure~\ref{rgb}.
    \item The 1-dimensional cells correspond to the left cosets of the form $wW_s$ for $s\in \{\red,\gr,\bl\}$. Figure~\ref{rgb} displays $wW_{\red}$ in \textcolor{Mred}{red}, $wW_{\gr}$ in \textcolor{Mgreen}{green}, and $wW_{\bl}$ in \textcolor{Mblue}{blue}.
    \item The 0-dimensional cells correspond to the left cosets of the form $wW_{\{s,t\}}$ for $s\neq t\in \{\red,\gr,\bl\}$. \rev{Figure~\ref{rgb} displays each $wW_{\{s,t\}}$ as a small node adjacent to $s$- and $t$-colored 1-cells.}
\end{itemize}
Then the left and right actions of $W$ on $\Com(W,S)$ are described as follows.
\begin{itemize}
    \item 
The left action of a generator $s\in S$ is given by reflection with respect to the original $s$-hyperplane (see Figure~\ref{rgbhyp}), for example the left action of $\red\in S$ is the reflection along the line containing the \textcolor{Mred}{red} wall of the 2-cell indexed by the identity element $e\in W$.
\item The right action of a generator $s\in S$ 
is given by reflection along the $s$-wall, for example the 2-dimensional cells $w,w\red\in W$ are reflections of each other with respect to the \textcolor{Mred}{red} wall between them.
\end{itemize}
Thus the core cosets of the form $W_t w W_s = w W_s$, where $s,t\in S$, are those $1$-cells on the original $t$-hyperplane; for example see Figure~\ref{rgbg} for core cosets $W_{\gr} w W_s = w W_s$, where $s\in S$ is not fixed. Then see Figure~\ref{rgbcore} for all core cosets\rev{. Every coset of the form $W_\emptyset w W_\emptyset$ is core, while the only core cosets of the form $W_IwW_J$ for $|I|=|J|=2$ in this example are the trivial cosets $W_IeW_I$.}
\end{ex}

\subsection{Regular and singular expressions}

A reduced expression $w=s_1\cdots s_\ell$ in $(W,S)$ is read from the left to right to give the sequence
\[[e,s_1,s_1s_2,\ldots ,s_1s_1\cdots s_\ell].\] 
We view the sequence as a path going through the coresponding maximal cells in $\Com(W,S)$.
Here $e\in W$ denotes the identity element in the group $W$ and marks the starting point of the path.
One observes that, when moving from one maximal cell to the next, the path actually goes through the (codimension one) face shared by the two maximal cells:
\begin{equation}\label{eq.regex}
[\{e\}\subset \{e,s_1\}\supset \{s_1\}\subset\{s_1,s_1s_2\}\supset\{s_1s_2\}\subset\cdots 
\supset\{w\}]
=[e, eW_{s_1}, s_1, s_1W_{s_2}, s_1s_2 ,\cdots
, w].
\end{equation}
Note that in every `$\supset$' we make a choice of going \emph{forward}: take the maximal element from the previous step, e.g., take $s_1s_2$ from $\{s_1,s_1s_2\}=s_1W_{s_2}$. 
To record \eqref{eq.regex}, it is therefore enough to give the following list of parabolic subsets.
\begin{equation}\label{eq.realregex}
    [\emptyset,\{s_1\},\emptyset,\{s_2\},\emptyset,\cdots,\{s_\ell\},\emptyset]
\end{equation}

\begin{ex}\label{ex.aA2cont}
We continue Example~\ref{ex.aA2}.
 The regular (reduced) expression $w=rbgrbg$ in $(\widetilde S_3,\{\red,\gr,\bl\})$ gives the sequence
\[[e,r,rb,rbg,rbgr,rbgrb,rbgrbg]\] viewed as the (reduced) path in Figure~\ref{rgbregex}.
\end{ex}

A main idea of singular Coxeter combinatorics (see \cite{Wthesis} and \cite{EKo}) is that a path can go through even smaller cells, and that a path can end not only at some $\{w\}$ but at any $wW_J$.
More generally, a path can start at $W_I=W_I e W_I$ instead of at $\{e\}=W_\emptyset e W_{\emptyset}$, in which case our path is a sequence of double cosets rather than of left cosets.
In this generality, a \emph{(singular) expression} in $(W,S)$ is a sequence
\begin{equation}\label{eq.multiexp}
M_{\bullet} = [[I_0 \subset K_1 \supset I_1 \subset K_2 \supset \ldots \subset K_m \supset I_m]]. \end{equation}
where $I_i,K_i\subset S$ are \emph{finitary} subsets, i.e., $W_{I_i},W_{K_i}$ are finite.  
The finitary restriction is necessary when assigning a path to \eqref{eq.multiexp}: denoting the unique Bruhat maximal element of a finite double coset $q$ by $\ma{q}$, the (forward) path for \eqref{eq.multiexp} is
\[p_\bullet=[[p_0,p_1,\cdots,p_{2m}]]\quad\text{where}\quad p_{2i+1} = p_{2i}W_{K_i} \text{ and }p_{2i} = \rev{W_{I_0}}\ma{p_{2i-1}}W_{I_i}\]
Here $p_{2i+1}$ is an $(I_0,K_i)$-coset and $p_{2i}$ is an $(I_0,I_i)$-coset, and by $(I,J)$-coset we mean $(W_I,W_J)$-double coset; see Section~\ref{ss.prelim} for more details.
The path $p_\bullet$ ends at the $(I_0,I_m)$-coset $p_{2m}$.
Thus, analogously to the regular case\footnote{To be more precise, the singular expressions \rev{correspond to products in the Coxeter monoid $(W,*)$ (see Section~\ref{ss:singmon}), not the Coxeter group $W$. }
For reduced expressions, there is no difference between the two variations.} where \eqref{eq.regex} expresses the element $w$, the expression \eqref{eq.multiexp} expresses the double coset $p_{2m}$. We write in this case
\[p_{2m}\expr [[I_0 \subset K_1 \supset I_1 \subset K_2 \supset \ldots \subset K_m \supset I_m]].\]

Williamson \cite{Wthesis} also introduces the notion of reduced expressions (resp., reduced paths) for double cosets. Here is a reformulation.
\begin{defn}\cite[Definition 1.4]{EKo}
An expression of the form \eqref{eq.multiexp} is a \emph{reduced expression} of $p\in W_{I_0}\backslash W/W_{I_m}$ if 
    \begin{equation} \label{eq.pmarex}
    \rev{\ma{p} = w_{K_1} w_{I_1}\inv w_{K_2} w_{I_1}\inv w_{K_3} \cdots  w_{I_{m-1}}\inv w_{K_m},}
    \end{equation}
 where $w_I$ denotes the longest elements of the finite parabolic subgroup $W_I$, and we have
\begin{equation*} \rev{\ell(\ma{p}) =  \ell(w_{K_1}) - \ell(w_{I_1})+ \ell(w_{K_2}) - \cdots -\ell(w_{I_{m-1}}) + \ell(w_{K_m}).}
\end{equation*}
\end{defn}
See Definition~\ref{def.rex} for Williamson's definition, in a singlestep formulation, and more details. 
Then a \emph{reduced path} is a forward path associated to a reduced expression in the sense of the previous paragraph.
Loosely, a reduced path is a path which heads away from the origin in the Coxeter complex. 
A joint work with Elias \cite{EKo} introduces and classifies the braid relations for double cosets (see Section~\ref{ss:singmon} and Section~\ref{s.switchbacks}) and proves Matsumoto's theorem, i.e., that any two reduced expressions of the same double coset are related by a series of braid relations. 

\begin{ex}\label{ex.rgbsing}
We are in the setting of Example~\ref{ex.aA2}. The regular reduced expression in Example~\ref{ex.aA2cont} corresponds to the singular reduced expression
\[\{rbgrbg\}\expr [[\emptyset\subset\{\red\}\supset\emptyset\subset\{\bl\}\supset\emptyset\subset\{\gr\}\supset\emptyset\subset\{\red\}\supset\emptyset\subset\{\bl\}\supset\emptyset\subset\{\gr\}\supset\emptyset]]\]
which records the colors the associated path passes through in Figure~\ref{rgbregex}. 
The following is a singular reduced expression of a regular element that does not come from a regular expression, for which the reader can draw a path on Figure~\ref{rgb}.
\[ \{ grgbr \} \expr [[\emptyset\subset \{\red,\gr\}\supset \{\red\}\subset \{\red,\bl\}\supset \emptyset]]\] 
Here is a singular reduced expression of $p=W_{\{g\}}grW_{\{g,b\}}$
\[ p \expr [[\{\gr\}\subset\{\red,\gr\}\supset\{\gr\}\subset\{\gr,\bl\}]],\] 
whose path consists of proper double cosets. 
\end{ex}

\subsection{Atoms and atomic expressions}

Now suppose $p$ is a core $(I,J)$-coset, for some finitary subsets $I,J\subset S$. 
Then $C_p$ is a maximal cell in $\Com(W,S,I)$.
It is thus natural to ask whether there exist \emph{$I$-regular} (expressions and) paths in $\Com(W,S,I)$ 
that terminate at $p$.
By an \emph{$I$-regular} expression we mean an expression \eqref{eq.multiexp} such that $|I_i|=|I|$ and $|K_i|=|I|+1$ hold, exactly like a regular expression \eqref{eq.realregex} does. Its 
path alternates between two types of moves:
from a maximal cell to its codimension one face and from a codimension one face to a maximal cell,
exactly like a regular path \eqref{eq.regex} does.
The answer is positive and appears in a joint work with Elias-Libedinsky-Patimo \cite{KELP3} (see Proposition~\ref{prop.arexfinns} or \cite[Corollary 2.17]{KELP3}).
In fact, \cite{KELP3} shows that there is a reduced expression with the above condition and calls it an 
\emph{atomic reduced expression}. To explan the terminology, let us start with atoms. 

\begin{defn}
    An \emph{atom} (or an \emph{atomic coset}) is 
a core coset of the form $W_{K\setminus s}w_{K}W_{K\setminus t}$, where $K\subset S$ is finitary and $s,t\in K$.
\end{defn}
Here $w_K$ denotes the maximal element in $W_K$, and it is straightforward that the core coset condition implies $t=w_Ksw_K$. 
A further justification of the definition (in addition to that they are the simplest core cosets) is given by the fact that an atom has a unique reduced expression (see Proposition~\ref{prop.atom}). 
An \emph{atomic (reduced) expression} then refers to a (reduced) composition of atoms, identifying the atoms with their unique reduced expressions. For example, if $\at=W_{I}w_{K}W_{J}$ 
and $\atb = W_Jw_LW_M$ are atoms, then 
their unique reduced expressions are
\[\at\expr [[I\subset K\supset J]],\quad \atb \expr [[J\subset L\supset M]]\]
and we have
\[ \at\circ \atb \expr [[I\subset K\supset J\subset L\supset M]].\]
Thus the forementioned result of \cite{KELP3} says that a core coset decomposes as a reduced composition of atoms.
In all above senses, an atom, in the study of (core) double cosets, plays the same role as what a simple reflection $s\in S$ plays in Coxeter combinatorics.

\begin{ex}
We continue Example~\ref{ex.aA2}.
\begin{enumerate}
\item  There are in total nine atomic cosets in $(\widetilde{S}_3,\{\red,\gr,\bl\})$, which are
\begin{itemize}
    \item three 2-cells $W_\emptyset rW_\emptyset= rW_\emptyset$, $W_\emptyset gW_\emptyset=gW_\emptyset$, $W_\emptyset gW_\emptyset=bW_\emptyset$;
    \item six 1-cells 
    \[W_g rgW_r=\textcolor{Mred}{rgW_r},\qquad W_brbW_r=\textcolor{Mred}{rbW_r},\]
    \[W_rgrW_g=\textcolor{Mgreen}{grW_g},\qquad W_bgbW_g=\textcolor{Mgreen}{gbW_g},\]
    \[W_rbgW_b=\textcolor{Mblue}{brW_b},\qquad W_gbgW_b=\textcolor{Mblue}{bgW_b};\] 
    \item no 0-cells, because if $I\subset S$ has order two then $Is=S$ is not finitary;
\end{itemize}
 displayed in Figure~\ref{rgbatom}.  
 \item The core double coset $p=W_g rgbrgbrgW_{r}=\textcolor{Mred}{rgbrgbrgW_{\red}}$ has an  expression
 \[p\expr [[\{\gr\}\subset\{\gr,\red\}\supset\{\red\}\subset\{\red,\bl\}\supset\{\bl\}\subset\{\gr,\bl\}\supset\{\gr\}\subset\{\gr,\red\}\supset\{\red\}]]\]
 which is atomic and is depicted in Figure~\ref{rgbarex}. Writing it alternatively as a composition of atoms,
 \[p \expr W_grgW_r\circ W_rbrW_b\circ W_bgbW_g\circ W_grgW_r\]
 as depicted in Figure~\ref{rgbarex2}.
\begin{figure}
\centering
\begin{minipage}{.5\textwidth}
  \centering
  \includegraphics[width=.95\linewidth]{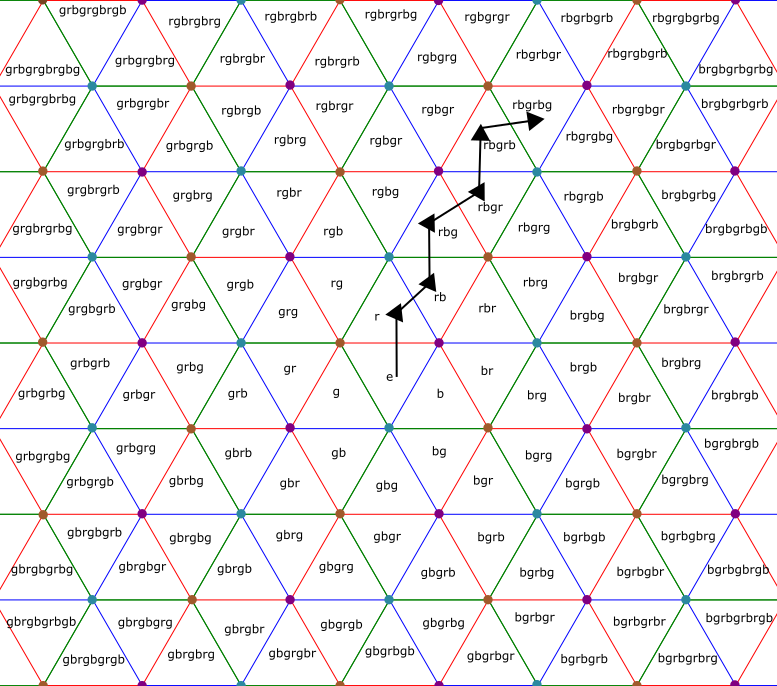}
  \captionof{figure}{A regular expression}
  \label{rgbregex}
\end{minipage}%
\begin{minipage}{.5\textwidth}
  \centering
  \includegraphics[width=.95\linewidth]{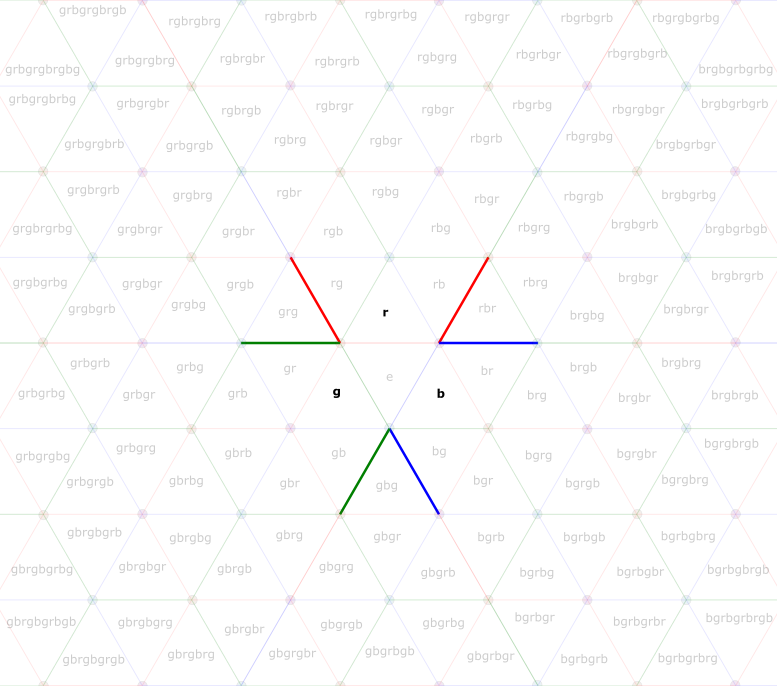}
  \captionof{figure}{Atoms in $(\widetilde{S}_3,\{\red,\gr,\bl\})$}
  \label{rgbatom}
\end{minipage}
\end{figure}
\begin{figure}
\centering
\begin{minipage}{.5\textwidth}
  \centering
  \includegraphics[width=.95\linewidth]{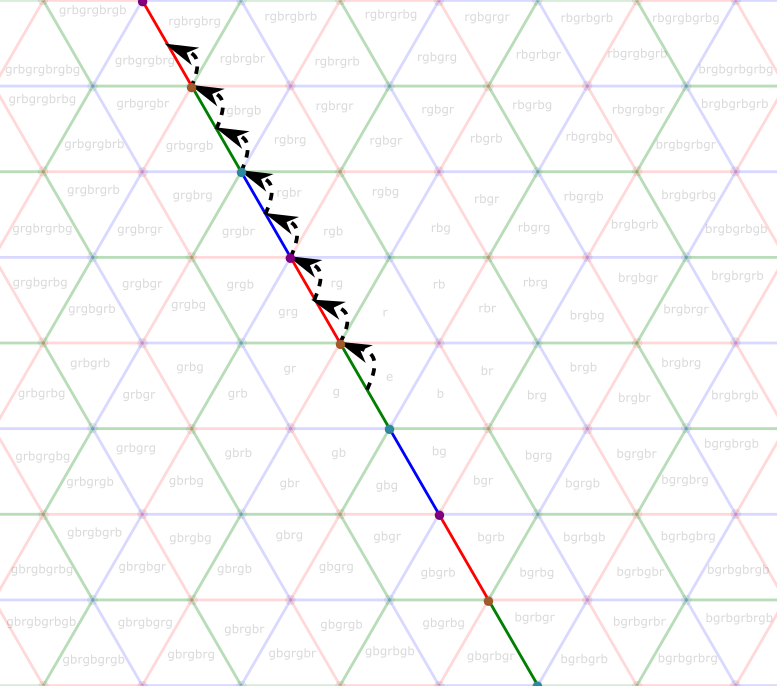}
  \captionof{figure}{An atomic expression of \\ $W_grgbrgbrgW_r=rgbrgbrgW_{\red}$}
  \label{rgbarex}
\end{minipage}%
\begin{minipage}{.5\textwidth}
  \centering
  \includegraphics[width=.95\linewidth]{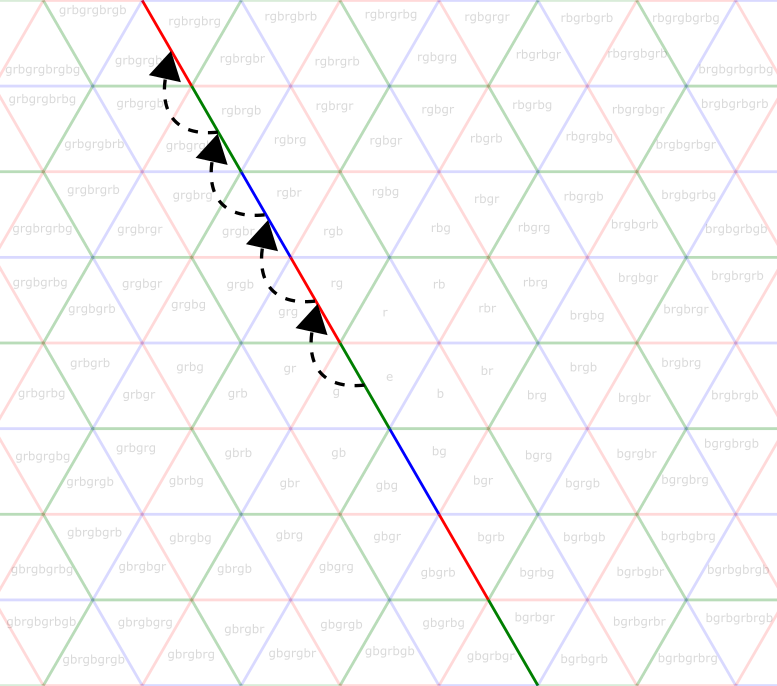}
  \captionof{figure}{The same expression \\ \rev{zoomed-out until 0-cells are invisible}}
  \label{rgbarex2}
\end{minipage}
\end{figure}

\end{enumerate}
\end{ex}

\subsection{Atomic braid relations and Matsumoto's theorem}\label{ss.intromats}
The purpose of the current work is to understand the relations between atomic expressions. For reduced expressions, we achieve this in the atomic Matsumoto theorem:

\begin{thm}[Theorem~\ref{mats}]\label{thm.matsintro}
Two atomic reduced expressions of the same core coset are related by \emph{atomic braid relations}. 
\end{thm}
We refer to Definition~\ref{def.braid} for its definition, but an atomic braid relation is of the form
\begin{equation}\label{eq.atomicbr}
\underbrace{\at \circ \atb' \circ \at''\cdots}_{m} = \underbrace{\atb \circ \at' \circ \atb''\cdots}_{m}     ,
\end{equation}
where $\at,\atb$ are atoms; the atoms $\at',\at'',\cdots$ and  $\atb',\atb'',\cdots$ are certain twists of the atoms $\at$ and $\atb$, respectively, determined by their position in the composition; and $m=m_{\at,\atb}\geq 2$ is an integer (not necessarily equal to one of the entries in the Coxeter matrix for $(W,S)$). See also Proposition~\ref{prop.braid} and Proposition~\ref{prop.suds}. Here is an example in type $D_4$ (our running example $\widetilde S_3$ is unfortunately too small to have nontrivial atomic braid relations) where no twists occurs; see Examples~\ref{ex.ab2},~\ref{ex.abE8}, Section~\ref{s.dihedral}, Section~\ref{ss.trihedral} for examples with twists.

\begin{ex}[Example~\ref{ex.ab4}]
Let $(W,S)$ be of type $D_4$ where $S=\{1,2,3,4\}$ is indexed so that $2$ is at the center of the Dynkin diagram. Let us use the shorthand $\{1,2,3\}=123$, etc. for subsets of $S$. Then
\[\at\expr[[13\subset 123\supset 13]],\qquad \atb\expr[[13\subset 134\supset 13]]\] are atomic cosets. Their atomic braid relation is
\[\at\circ\atb\circ\at\circ\atb\expr \atb\circ\at\circ\atb\circ \at,\]
where the left hand side is identified with the expression
\[[[13\subset 123\supset 13\subset 134\supset 13\subset 123\supset 13 \subset 134\supset 13]]\]
and similarly for the right hand side.
\end{ex}

We point out a key feature of our braid relations, a feature the braid relations from \cite{EKo} do not enjoy: the two sides of \eqref{eq.atomicbr} are of the same width. In particular, Theorem~\ref{thm.matsintro} implies that every atomic reduced expression of a coset has the same width.
Due to this feature, the combinatorics of atomic reduced expressions very much resembles that of regular reduced expressions in a Coxeter group, bringing us a substantial computational and pedagogical advantage over the presentation in \cite{EKo}.

\subsection{An atomic presentation}\label{ss.intatomicpres}
The non-reduced (atomic or not) expressions depend on different natural ways of composing double cosets in a non-reduced way. We consider two of them, namely the singular Coxeter monoid $\SC$ studied in \cite{EKo} (see Section~\ref{ss:singmon}) and the nilCoxeter algebroid $\Dem$ studied in \cite{KELP1} (see Section~\ref{ss.dem}). 
As is implicit above, we use the adjective  `singular' to mean `of parabolic double cosets'. 
In particular $\SC$ is the double coset analogue of the Coxeter monoid, and $\Dem$ is that of the nilCoxeter algebra, sometimes called the nilHecke algebra (our notation comes from $\Dem$ being equivalent to the category of singular Demazure operators, i.e., Demazure operators associated to parabolic double cosets; see \cite{Demazure,KELP1}). The two settings differ only in the quadratic relations, namely, $ss=s$ in the Coxeter monoid and $ss=0$ in the nilCoxeter algebra.
As can be guessed from the quadratic relation, it is easier to deal with non-reduced expressions in $\Dem$: they are equal to zero. 

As our second main result, we provide in Theorem~\ref{thm.presentDemazure} a presentation of the core cosets in $\Dem$ by generators and relations. Theorem~\ref{thm.presentDemazure} \rev{generalizes} the (nil)Coxeter presentation of Demazure operators \cite[Theoreme 1 and Proposition 3]{Demazure}. 

\subsection{Why core cosets?}

We restricted ourselves to core double cosets in Sections~\ref{ss.intromats},\ref{ss.intatomicpres}. Let us justify \rev{this}.
Let $p$ be an $(I,J)$-coset. The \emph{left (resp. right) redundancy} of $p$ is the subset
\begin{equation*}\label{introeq.red}
    \leftred(p) = I\cap \mi{p}J\mi{p}\inv\subset I,\quad \text{resp.,}\quad \rightred(p) = \mi{p}\inv I\mi{p} \cap J\subset J,
\end{equation*}
where $\mi{p}$ denotes the unique minimal element in $p$ (see Section~\ref{ss.prelim} for more details).
We have $W_{\leftred(p)} = W_I\cap \mi{p}W_J\mi{p}$ (see e.g., \cite[Lemma 2.12]{EKo}), which says that the redundancy encodes the redundancy in combining the left action and the right action of $W_I$ and $W_J$. 
In particular, the coset $p$ is core if and only if $I=\leftred(p)$ and $J=\rightred(p)$. 
Then a result from \cite{EKo} (see Proposition~\ref{prop.lowroad}) says that $p^{\core}=W_{\leftred(p)}\mi{p}W_{\rightred(p)}$ is a core $(\leftred(p),\rightred(p))$-coset and our coset $p$ has a reduced expression of the form 
\begin{equation}\label{introeq.lowroad}
[[I\supset \leftred(p)]]\circ M_\bullet \circ [[\rightred(p)\subset J]]    
\end{equation}
where $M_\bullet$ is a reduced expression for $p^{\core}$. 
This allows us to reduce problems on arbitrary double cosets to that on core cosets. 

For example, if an $(I,J)$-coset $p$ corresponds to a cell in $\Com(W,S,I)$, then we have $\leftred(p)=I$, $\rightred(p)=\mi{p}\inv I\mi{p}\subset J$, and $p^{\core}=W_I\mi{p}W_{\rightred(p)}$. The cell $p$ is not maximal in $\Com(W,S,I)$ if and only if the last inclusion $\rightred(p)\subset J$ is strict.
In this case $p$ is a face of the maximal cell $p^{\core}$ in $\Com(W,S,I)$.
Therefore, if we take a reduced path $p_\bullet$ in $\Com(W,S,I)$ for $p^{\core}$ then we obtain such a path for $p$ by adding to $p_\bullet$ the final extra step of restricting to a face. This extra step is the postcomposition $-\circ [[\rightred(p)\subset J]]$ in \eqref{introeq.lowroad}.
 
\subsection{Motivation from Lie theory}
Now that we have given purely combinatorial motivations for our work, let us explain our original motivations from the study of Hecke categories. 
In a joint work \cite{KELP4} with Elias, Libedinsky, Patimo, we construct a cellular-like basis of the singular Hecke category (also known as the \rev{category of }singular Soergel bimodules \cite{SingSb, Wthesis}) associated to $(W,S)$. The latter is a categorification of the Hecke algebroid and is the double coset analogue of the regular Hecke category (Soergel bimodules).
In constructing our basis elements, called \emph{light leaves}, a crucial step is to find desirable (with conditions depending on the situation) reduced expressions of double cosets.
The result of \cite{EKo} providing reduced expressions of the form \eqref{introeq.lowroad} is one main ingredient for \rev{this} step. Another important ingredient is to find a reduced expression $M_\bullet$ of a core coset.
In type $A$ and type $B$, as proved in \cite{KELP3}, the atomic reduced expressions are exactly the same as the regular reduced expressions in type $A$ and type $B$ respectively, and thus finding $M_\bullet$ reduces to the well-understood problem of finding a regular reduced expression of an element in a Coxeter group (of type $A$ and $B$).
This vastly simplifies the basis construction in finite and affine types $A$ and $B$.
In fact, the latter is an original motivation of \cite{KELP3} for initiating the atomic theory of Coxeter systems.

\subsection{Tits cone intersections and applications}

Finally, we mention an unexpected (or, expected but unexpectedly strong; see \cite[Remark 1.17]{EKo}) connection to a work of Iyama-Wemyss \cite{IW}, discovered in the final stage of the project. 
\cite{IW} introduces a new combinatorial structure $\cone(W,S,I)$, associated to a Coxeter system $(W,S)$ and a subset $I\subset S$, which they call a \emph{Tits cone intersection}.
They develop \rev{the} combinatorics of Tits cone intersections in affine type $ADE$ to describe the tilting theory for contracted preprojective algebras (certain idempotent summands of preprojective algebras) as well as the noncommutative resolutions for compound Du Val singularities, with eventual applications to birational geometry.   
\rev{T}his $\cone(W,S,I)$ corresponds exactly to $\Com(W,S,I)$ discussed above (under the construction of $\Com(W,S)$ as a quotient of the Tits cone; see Section~\ref{s.IW}). 
In particular, the chambers in $\cone(W,S,I)$ are labeled by the core $(I,J)$-cosets (see Proposition~\ref{prop.IW}).
When we explain the connection in Section~\ref{s.IW}, we work in the Tits cone instead of the Coxeter complex for a more convenient comparison to \cite{IW}, but the explanation is implicit in Section~\ref{ss.com} above and in \cite[Part I]{IW}.

This provides $\cone(W,S,I)$ with additional information obtained in the current paper, in particular, the atomic braid relations, the atomic Matsumoto theorem, and a presentation by generators and relations (see 
Proposition~\ref{prop.basis}).
Conversely, the results in \cite{IW} shed a new light on the singular Coxeter monoid $\SC$ and the nilCoxeter algebroid $\Dem$. 
We do not explore such consequences in the current paper, but content with 
referring to \cite[Theorem 0.5 and Section 4.2]{IW} for a beautifully illustrated classification of the arrangements of $(I,J)$-core cosets, where $I$ is fixed of order $|S|-3$ and $S$ is of affine type $ADE$.

\subsection*{Organization of the paper}

Section~\ref{s.prelim} introduces some algebraic structures commonly associated to a Coxeter system: Section~\ref{ss.prelim} recalls basic definitions and fix notation; Section~\ref{ss:singmon} recalls the singular Coxeter monoid $\SC$; Section~\ref{ss.dem} recalls the nilCoxeter algebroid $\Dem$.
Section~\ref{s.atom} introduces core cosets and the atom(ic coset)s and state necessary facts from \cite{KELP3}.
Section~\ref{s.switchbacks} explains the switchback relations from \cite{EKo} and  how they gives rise to atomic braid relations.
Section~\ref{s.mats} proves the atomic Matsumoto theorem.
Section~\ref{s.nonbraid} discusses atomic non-braid relations and give a presentation by generators and relations of the nilCoxeter algebroid.
In Section~\ref{s.dihedral} and Section~\ref{s.highrank}, we consider the substructures $\coreJ{J}$ of $\SC$ and $\Dem^{\core}_J$ of $\Dem$ consisting of the $(I,J)$-core cosets and study their combinatorics.
Section~\ref{s.dihedral} completely describe $\coreJ{J}$ and $\Dem^{\core}_J$ when $|S\setminus J|=2$; Section~\ref{s.highrank} discusses the weak Bruhat order on $\coreJ{J}$, resp., $\Dem^{\core}_J$, and computes some examples when $|S\setminus J|>2$.
Section~\ref{s.IW} relates our $\coreJ{J}$ and $\Dem^{\core}_J$ with \cite{IW}'s Tits cone intersection.

\subsection*{Acknowledgements.}
The main ideas in the paper came from discussions with Ben Elias, Nicolas Libedinsky and Leonardo Patimo during our joint expedition in the singular land; additional thanks to ELP for agreeing about posting the current paper before its prequel \cite{KELP3}.
We thank Nicolas Libedinsky for providing the background image for the figures in the introduction.
We thank Valentin Buciumas for helpful comments on the previous version.
The technical portion of the project benefited from examples of atomic expressions computed on \rev{SageMath \cite{sage}.}
We thank the referee for helpful comments. The author was partially
supported by the Swedish Research Council.

\part{Preliminaries}
\section{Singular Coxeter presentations}\label{s.prelim}

In this section we collect some basics of Coxeter combinatorics and recall two presentations of parabolic double cosets by generators and relations, from \cite{EKo} and \cite{KELP1} respectively.
Let $(W,S)$ be a Coxeter system. 

\subsection{Coxeter group preliminaries}\label{ss.prelim}

Given an element $x\in W$ the (Coxeter) \emph{length} $\ell(x)$ is the smallest length of an expression of $x\in W$, i.e., the smallest $\ell$ where $x=s_1s_2\cdots s_\ell$ for some $s_i\in S$.
A product $xy$, for $x,y\in W$, is said to be \emph{reduced} if $\ell(xy)=\ell(x)+\ell(y)$. In this case we write $xy=x.y$.
An expression $w=s_1s_2\cdots s_m$, where $s_1,s_2,\cdots,s_m\in S$, is called a \emph{reduced expression} if every product involved in computing the expression, in any order, is reduced.
The latter condition is equivalent to the product
\[(s_1\cdots s_i)s_{i+1}\]
being reduced for each $i=1,\cdots, m-1$, which also is equivalent to the length $m$ of the expression being the shortest possible, i.e., $m=\ell(x)$.
The \emph{left (resp., right) descent} set of $w\in W$ is defined and denoted as
\[\leftdes(w) = \{s\in S\ |\ sw\text{ is not reduced}\},\quad \text{ resp., }\quad \rightdes(w)=\{s\in S\ |\ ws\text{ is not reduced}\}.\]

A subset $I\subset S$ is said to be \emph{finitary} if the parabolic subgroup $W_I$ generated by $I$ is finite.
Given finitary $I,J\subset S$\rev{, we call an element of $W_I\backslash W/W_J$ an $(I,J)$-coset.}

\rev{An $(I,J)$-coset $\pcoset$ has a unique maximal element, denoted by $\ma{p}$, and a unique minimal element, denoted by $\mi{p}$, with respect to the Bruhat order (see, e.g., \cite[Lemma 2.12]{EKo}).
The \emph{left (resp. right) redundancy} of an $(I,J)$-coset $\pcoset$ is 
\begin{equation}\label{eq.red}
    \leftred(p) = I\cap \mi{p}J\mi{p}\inv,\quad \text{resp.,}\quad \rightred(p) = \mi{p}\inv I\mi{p} \cap J.
\end{equation}}
Howlett's theorem \cite[Lemma 2.12]{EKo} says that, when writing $w\in p = W_I\mi{p}W_J$ as a product $x\mi{p} y$, the choice of elements $x\in W_I$, $y\in W_J$ is unique up to the action of the redundancy, that is, the surjective map
    \begin{equation} 
        \begin{aligned}
            W_I\times W_J  &\to p  \\
            (x,y) &\mapsto x\mi{p}y 
        \end{aligned}
    \end{equation}
induces the bijections
\begin{equation}\label{eq.Howl}
(W_I/W_{\leftred(p)})\times W_J\to p\quad\text{ and }\quad W_I\times (W_{\rightred(p)}\backslash W_J)\to p.
\end{equation}
\cite[Lemma 2.12]{EKo} also says that, if $x\in W_I$ is a minimal representative in $xW_{\leftred(p)}$ and $y\in W_J$, or if $x\in W_I$ and $y\in W_J$ is minimal in $W_{\rightred (p)}y$, then we have $x\mi{p}y=x.\mi{p}.y$. 

\subsection{Coxeter monoids}\label{ss:singmon}

This subsection briefly introduces the regular and singular Coxeter monoids. We refer to  \cite[Section 1.3, Section 2]{EKo} for details.

The \emph{Coxeter monoid} (also known as the \emph{0-Hecke monoid} or the \emph{star monoid}) $(W,S,*)$ is defined as the monoid with the following presentation by generators and relations. The generator is the set $S$; the generating relations are
\begin{itemize}
    \item 
the $*$-quadratic relation for $s \in S$ 
\[s*s = s;\] 
\item
the braid relation 
for $s,t\in S$ with $m=m_{st}<\infty$
\[\underbrace{s * t * \cdots}_{m} = \underbrace{t * s * \cdots}_{m} .\]\end{itemize}
An expression $s_1*s_2*\cdots *s_k$ is \emph{reduced} if its width $k$ is minimal among that of the expressions of the same element in $(W,S,*)$. 
If $s*t*\cdots *u$ is a reduced expression in $(W,*)$, then $st\cdots u$ is a reduced expression in the Coxeter group $W$, and we identify the element $w=s*t*\cdots *u\in (W,*)$ with the element $w=st\cdots u\in W$. 
This identification imports the length function $\ell$ on the Coxeter group $W$ to the Coxeter monoid $(W,*)$ and justifies the following (abuse of) notation: we write $x*y=x.y$ in $(W,*)$ if $\ell(xy)=\ell(x)+\ell(y)$.

\begin{defn}
The \emph{singular Coxeter monoid} is the category $\SC=\SC(W,S)$ whose
\begin{itemize}
    \item  objects are the finitary subsets $I\subset S$;
    \item morphisms, from $I$ to $J$, are the $(I,J)$-cosets, i.e., $\Hom_\SC(I,J) = W_I \backslash W/W_J$;
    \item composition `$*$', of $\pcoset\in W_I \backslash W/W_J$ and $\qcoset\in W_J \backslash W/W_K$, is given by 
    \[p*\qcoset = W_I (\ma{p}*\ma{q}) W_K,\]
    for finitary subsets $I,J,K\subset S$.
\end{itemize}
\end{defn}

The category $\SC$ is given a presentation by generators and relations in \cite[Theorem 5.31]{EKo}. 
The generators are the double cosets of the form $W_IeW_{Is}$ and $W_{Is}eW_{I}$, where $Is:=I\sqcup s$ is our notation and $Is\subset S$ is finitary.
The composition of these generators are expressed in two ways. 
One way is to list the finitary subsets appearing:
\begin{equation}\label{IInotation}
    [I_0,I_1,I_2,\cdots ,I_m] := (W_{I_0}eW_{I_1})*(W_{I_1}eW_{I_2})*\cdots *(W_{I_{m-1}}eW_{I_m})
\end{equation}
The other way is to list the differences between the finitary subsets:
\begin{equation}\label{pmnot}
[I \pm_1 s_1 \pm_2 s_2 \cdots \pm_m s_m] := [I_0,I_1,I_2\cdots ,I_m]    
\end{equation}
Here $\pm_i\in\{+,-\}$ and $s_i$ are such that, either $I_is_i=I_{i-1}$ and $\pm_i = -$ or $I_i=I_{i-1}s_i$ and $\pm_i = +$. We call $m$ the \emph{width} of the expression \eqref{pmnot}. (See \cite[Section 3.4]{EKo} for discussions on the \emph{length} of an expression, which \rev{we will not need.)}

Under notation \eqref{pmnot}, the generating relations are 
\begin{itemize}
    \item the \emph{up-up relations} \[[I+s+t] \expr [I+t+s] ;\]
\item the \emph{down-down relations} 
\[[I-s-t] \expr [I-t-s];\]
\item the \emph{switchback relations} 
\[ [I + \sblue - \teal] \expr [I - \textcolor{black}{u_1}+\textcolor{RoyalBlue}{u_0} -\textcolor{black}{u_2} +\textcolor{black}{u_1} -u_3+\textcolor{black}{u_2} \ \ldots\ -u_{d-1}+u_{d-2} -\textcolor{Green}{u_d}+u_{d-1}];\]
\item the \emph{$*$-quadratic relations} 
\[[I-s+s] \expr [I];\]
\end{itemize}
where \rev{$I,s,t$ are such that the expressions are well-defined, e.g., $s,t\not\in I$ and $Ist\subset S$ is finitary for the up-up relation. The only additional condition on $I,s,t$ is that a switchback relation requires $w_{Is}\sblue w_{Is} \neq \teal$.}
The elements $u_i\in I\sblue$ are explained in Section~\ref{s.switchbacks} in detail (see in particular Definition~\ref{def:useq} where we call $u_\bullet$ the rotation sequence associated to $(I\sblue,\sblue,\teal)$).
The first three classes of relations are called the \emph{braid relations}. 

\begin{defn}\label{def.p*q}
For finitary $I,J,K\subset S$, an $(I,J)$-coset $\pcoset$, and a $(J,K)$-coset $\qcoset$, the composition $\pcoset*\qcoset$ is said to be \emph{reduced} if $(\ma{\pcoset}w_J\inv).\ma{\qcoset}$, or equivalently $\ma{p}.(w_J\inv \ma{q})$, is reduced. 
\end{defn}

\begin{defn}\label{def.rex}
An expression 
\[I_\bullet = [I_0,I_1,\cdots,I_m]=(W_{I_0}eW_{I_1})*(W_{I_1}eW_{I_2})*\cdots *(W_{I_{m-1}}eW_{I_m})\] is \emph{reduced} if the composition
of the $(I_0,I_{i})$-coset $\pcoset_i = (W_{I_0}eW_{I_1})*(W_{I_1}eW_{I_2})*\cdots *(W_{I_{i-1}}eW_{I_i})$ and the $(I_i,I_{i+1})$-coset $W_{I_{i}}eW_{I_{i+1}}$
is reduced for each of $i=1,\cdots m-1$. 
If $1\leq i \leq m-1$ is such that $I_i\supset I_{i+1}$ then the latter condition at $i$ is automatic; if $I_i\supset I_{i+1}$ then reducedness at $i$ is equivalent to the two conditions $\mi{p_i}=\mi{p_{i+1}}$ and $\leftred(p_i)=\leftred(p_{i+1})$. 
\end{defn}

The condition in Definition~\ref{def.rex} is equivalent to the more symmetric condition 
\[I_\bullet =(W_{I_0}eW_{I_1}).(W_{I_1}eW_{I_2}).\cdots .(W_{I_{m-1}}eW_{I_m}),\]
by which we mean every composition that can be involved in the right hand side, in all possible order of composition, is reduced.
See \cite[Section 3]{EKo} for other equivalent criteria and explanations; our formulation in Definition~\ref{def.rex} in terms of redundancy is the original definition from Williamson~\cite{Wthesis}.

\rev{
The anti-involution $w\mapsto w\inv$ on $W$ induces the contravariant involution on $\SC$ which fixes the objects and sends $p = W_I w W_J$ to $p\inv:=W_J w\inv W_I$.
This is compatible with reversing expressions, resp., reduced expressions, as in
\begin{equation*}
    [I_0,I_1,\cdots,I_m]\inv := [I_m,\cdots,I_1,I_0].
\end{equation*}
\begin{prop}\cite[Proposition 4.7]{EKo}\label{prop.inv}
Let $I_\bullet\expr p$. Then $(I_\bullet)\inv\expr p\inv$, and $I_\bullet$ is reduced if and only if $I_\bullet\inv$ is reduced.
\end{prop}

We end the subsection with the following analogue of Matsumoto's theorem.
\begin{thm}\cite[Theorem 5.30]{EKo}]\label{thm.mats}
 Every reduced expression of an $(I,J)$-coset is related by the braid relations, namely, by the upup, downdown, and switchback relations.   
\end{thm}
}

\subsection{NilCoxeter algebras}\label{ss.dem}

This subsection introduces the nilCoxeter algebras and nilCoxeter algebroids. We follow \cite[Section 3]{KELP1} (see also \cite[Remark 2.25]{KELP3}) and refer to it for all omitted details.

The \emph{nilCoxeter algebra} $\NCA(W,S)$ associated to a Coxeter system $(W,S)$ is the $\kk$-linear algebra presented by generators and relations as follows. 
The set of generators is $\{\partial_s\ |\ s\in S\}$;
the generating relations are
\begin{itemize}
    \item the Coxeter braid relations, that is, for $s,t\in S$ with $m=m_{st}<\infty$
\[\underbrace{\partial_s  \partial_t  \cdots}_{m} = \underbrace{\partial_t  \partial_s \cdots}_{m} \]
\item the \emph{nil-quadratic relations}, that is, for $s\in S$ 
\[\partial_s \partial_s = 0.\]
\end{itemize}
Then $\NCA(W,S)$ has a $\kk$-basis indexed by the elements in $(W,S)$, or equivalently the elements in $(W,S,*)$, and moreover has the same reduced expressions as for the Coxeter group $(W,S)$ or the Coxeter monoid $(W,S,*)$.
In fact, $\NCA(W,S)$ is the associated graded \rev{algebra} of the group algebra $\kk W$ (resp., the linearization of the Coxeter monoid $(W,*)$) filtered with respect to the length function.

\begin{defn}
The \emph{nilCoxeter algebroid}, or the \emph{singular Demazure algebroid}, $\Dem=\Dem(W,S)$ is the $\kk$-linear category whose 
objects are the finitary subsets $I\subset S$ 
and whose morphisms and their composition are defined via the following presentation by generators and relations.
\begin{itemize}
        \item The generating morphisms are \begin{equation}\label{eq.demgen}
            \partial_{[I,Is]}:Is\to I\quad\text{and}\quad \partial_{[Is,I]}:I\to Is,
        \end{equation}
        for $Is\subset S$ finitary. That is, a morphism in $\Dem(W,S)$ is a $\kk$-linear combination of compositions of some morphisms of the form \eqref{eq.demgen}. We denote such compositions as expressions as in \eqref{IInotation} and \eqref{pmnot}, for example, $\partial_{[I+s-s]}=\partial_{[I,Is,I]}:=\partial_{[I,Is]}\circ \partial_{[Is,I]}$.
        \item The generating relations are
        \begin{itemize}
\item the \emph{up-up and down-down relations} \[\partial_{[I+s+t]} = \partial_{[I+t+s]}, \quad
\partial_{[I-s-t]} = \partial_{[I-t-s]};\]
\item the \emph{switchback relation} 
\[ \partial_{[I + \sblue - \teal]} = \partial_{[I - \textcolor{black}{u_1}+\textcolor{RoyalBlue}{u_0} -\textcolor{black}{u_2} +\textcolor{black}{u_1} -u_3+\textcolor{black}{u_2} \ \ldots\ -u_{d-1}+u_{d-2} -\textcolor{Green}{u_d}+u_{d-1}]};\]
\item the \emph{nil-quadratic relation}
\[\partial_{[I-s+s]}= 0.\]
        \end{itemize}
\end{itemize}
The subset $I\subset S$ and elements $s,t\in S$ involved in the generating relations are only required to make the expressions well-defined, and $u_\bullet\in S$ is the rotation sequence for $(I\sblue,\sblue,\teal)$ (see Definition~\ref{def:useq}).
\end{defn}
In other words, $\Dem(W,S)$ has the same presentation as (the lin\rev{e}arization of) $\SC(W,S)$ except it satisfies 
the nil-quadratic relation instead of the $*$-quadratic relation. 
\cite[Theorem 3.21]{KELP1} shows that $\Dem(W,S)$ agrees with the category of Demazure operators (on the symmetric algebra of a reflection faithful and balanced realization). 
In particular, we have
\begin{equation}\label{eq.kSC=SD}
\Hom_{\Dem(W,S)}(I,J) = \langle \partial_\pcoset\rangle_{\pcoset\in W_I\backslash W/W_J} =\kk\Hom_{\SC(W,S)}(I,J),   
\end{equation}
for each morphisms space in $\Dem(W,S)$ where the set $\{\partial_p\ |\ \pcoset\in W_I\backslash W/W_J\}=\Hom_{\SC(W,S)}(I,J)$ is a $\kk$-basis.
Moreover, each basis element $\partial_\pcoset$ has the same reduced expressions as $\pcoset\in W_I\backslash W/W_J$ in $\SC(W,S)$ does.

\section{Core and atoms}\label{s.atom}

This section contains no original content and is based on \cite[Section 2]{KELP3}. 

\begin{defn}\label{def.core} Let $p$ be an $(I,J)$-coset. We call $p$ a \emph{core coset}\footnote{Note that we call a core coset \emph{a coset with full redundancy} in \cite{EKo} but switch to the current terminology in all other papers.}  if $\leftred(p) = I$ and $\rightred(p) = J$, or equivalently if $I = \mi{p} J \mi{p}^{-1}$. \end{defn}

It follows from \eqref{eq.red} and the suceeding discussion that Definition~\ref{def.core} is equivalent to Definition~\ref{introdef.core}.
The following result allows us to reduce problems on double cosets to that on core cosets.

\begin{prop}\cite[Lemma 4.27 and Proposition 4.28]{EKo}\label{prop.lowroad}
    Given an $(I,J)$-coset $p$, the $(\leftred(p),\rightred(p))$-coset $p^{\core}:=W_{\leftred(p)}\mi{p}W_{\rightred(p)}$ is a core coset. Moreover, if $p^{\core}\expr M_\bullet$ is a reduced expression then 
    \[p\expr[[I\supset \leftred(p)]]\circ M_\bullet \circ [[\rightred(p)\subset J]]\]
    is a reduced expression.
\end{prop}

Here is an important observation about reduced expressions of a core coset and of a general coset.
\begin{lem}\cite[Lemma 4.22, 4.25]{EKo}\label{lem.redconst}
Let $p$ be an $(I,J)$-coset with a reduced expression $p\expr [I_0,\cdots, I_m]$. Letting $p_i\expr[I_0,\cdots,I_i]$, the sequence $\leftred(p_i)$ is decreasing. If furthermore $p$ is core, then $I=\leftred(p_i)$ for all $i\leq m$.
\end{lem}

\cite{KELP3} proves a number of additional facts on core cosets. In particular, composition of core cosets can be understood as follows.

\begin{lem}\label{corecore}
    Let $p$ be a $(I,J)$-coset and $q$ be a $(J,K)$-coset and suppose  $r=p*q$ is reduced.
    If $p$ and $q$ are core cosets, then $r$ is a core coset.
     Conversely, if $r$ is a core coset and $|J|=|I|=|K|$, then $p$ and $q$ are core cosets.
\end{lem}
\begin{proof}
See \cite[Lemma 2.10]{KELP3} for the first claim and \cite[Proposition 2.14]{KELP3} for the converse.
\end{proof}

Now we introduce the atoms of this paper.

\begin{defn}\cite[Definition 2.18]{KELP3}\label{def.atom}
    A coset is called an \emph{atomic coset} (or an \emph{atom}) if it is a core coset and has a reduced expression of the form $[I+s-t]$.
\end{defn}

We emphasize the core condition in Definition~\ref{def.atom}. 

\begin{rem}\label{rem.I,s=atoms}
A coset of the form $\pcoset\expr[I+s-t]$ is core exactly when $t = w_{Is} s w_{Is}$. 
Thus the atoms in $(W,S)$ are indexed by the pairs $(I,s)$, with $sI\subsetneq S$ finitary and $s\in S\setminus I$.   
\end{rem}

\begin{rem}\label{rem.inverseatom}
    If $\at\expr[I,Is,J]=[I+s-t]$ is an atom, then its \rev{\emph{reverse}} $\at\inv\expr [J,Is,I]=[J+t-s]$, where $t=w_{Is}s w_{Is}$, is an atom. (See \cite[Section 4.4]{EKo} for a general discussion on \rev{reversing} double cosets.)
\end{rem}

Another way to characterize an atom among the cosets of the form $\pcoset\expr[I+s-t]$ is that an atom does not allow for braid relations:


\begin{prop}\cite[Proposition 5.14]{EKo}\label{prop.atom}
An atomic coset has a unique reduced expression. More precisely, a coset of the form $\pcoset\expr[I+s-t]$ is atomic if and only if $\pcoset$ has a unique \rev{reduced} expression; when $\pcoset$ is not atomic the switchback relation applies.
\end{prop}

\rev{We are interested in the composition of atoms, say 
\begin{equation}\label{eq.a1a2am}
    \at_1*\at_2*\cdots *\at_m.
\end{equation}
Replacing each $\at_i$ with its unique reduced expression given in Proposition~\ref{prop.atom}, we view the composition~\eqref{eq.a1a2am} as an expression in $\SC$. Expressions in $\SC$ arising in this way are characterized as follows.}

\begin{defn}
An expression of the form
\begin{equation}\label{eq.atomicex}
[L_0,I_0,L_1,I_1,\cdots I_{m-1},L_m]=[L_0+s_1-t_1+\cdots +s_m-t_m]    
\end{equation}
 is called an \emph{atomic expression} if, for each $0\leq i<m$, the coset 
 \begin{equation}\label{eq.at=}
 \at_i\expr [L_i,I_i,L_{i+1}]=[L_i+s_{i+1}-t_{i+1}]    
 \end{equation}
  is core (thus an atom).
\end{defn}

\rev{A}n atomic expression is a composition of atoms in the same way a regular expression in a Coxeter group is a product of simple reflections.
In this sense, the following result from \cite[Section 2.5]{KELP3} says that every core coset decomposes into atoms. We slightly reformulate the result and include a proof for convenience. 

\begin{prop}\label{prop.arexfinns}
    Let $\pcoset$ be a core $(I,J)$-coset. Then $p$ has an atomic reduced expression. In fact,
    \begin{enumerate}
    \item For each $\sblue\in \leftdes(\ma{\pcoset})\setminus I$, there is a reduced atomic expression of the form
    \[\pcoset\expr [I+\sblue-t+\cdots].\] 
    \item For each $\sblue\in \rightdes(\ma{\pcoset})\setminus J$, there is a reduced atomic expression of the form 
    \[\pcoset\expr [I+\cdots + u-\sblue].\]     
    \end{enumerate} 
\end{prop}
\begin{rem}\label{rem.[I]}
    Note that
$\leftdes(\ma{\pcoset})=I$ implies $\ma{\pcoset}=w_I\in W_I$ and thus $\pcoset$ is a trivial $(I,I)$-coset with the empty atomic expression $[I]$. 
\end{rem}
\begin{proof}
We prove the first claim by induction on $\ma{p}$. The second claim has a similar proof. 
The base case is given by Remark~\ref{rem.[I]}, which also allows us to assume that $I \subsetneq \leftdes(r)$. 

Let $s \in \leftdes(r) \setminus I$. 
\cite[Proposition 4.21]{EKo} gives a reduced expression $r\expr [I,Is, \ldots, J]$. Let $n$ be the $(Is,J)$-coset such that $r = [I,Is] . n$.
The left redundancy of $n$ is at most the size of $J$, so it is a proper subset of $Is$. By Proposition~\ref{prop.lowroad}, for any $t \in Is \setminus \leftred(n)$, the coset $n$ has a reduced expression of the form $[Is, Is \setminus t, \ldots J]$. Let $q$ be the $(Is \setminus t,J)$-coset such that $n = [Is,Is\setminus t] . q$.
Then
$ r = \at . q$ for $\at \expr[I,Is, Is \setminus t]$. Then $\at$ is an atom and $q$ is a core coset by Lemma~\ref{corecore}. Since $\ma{q}<\ma{p}$ (see Definition~\ref{def.p*q}), the induction step is established.
\end{proof}

\part{Atomic relations}
\section{Switchback relations versus atomic braid relations}\label{s.switchbacks}

In this section we fix a finite Coxeter system $(W,S)$ and employ Notation~\ref{notation.prime}. 
We recommend the reader to have in mind that the discussion in this section is to be applied in the setting where we have a not-necessarily-finite Coxeter system $(\widetilde W,\widetilde S)$ and $S\subset \widetilde S$ is a finitary subset generating the finite Coxeter subgroup $W$ in $\widetilde W$. 

\begin{notation}\label{notation.prime}
For $s\in S$, we write $\con{s} = w_S\inv s w_S=w_S s w_S$. For $L\subset S$, we write $\Lcon = \{\con{s}\ |\ s\in L\}$. 
\end{notation}

\rev{Note that $\con{s}\in S$ and that the bijection $s\mapsto \con{s}$ on $S$ is the restriction of the automorphism $w\mapsto w_Sww_S$ on $(W,S)$. 
}

Our convention throughout the section is that  $\sblue,\teal\in S$ are two elements with the condition $\sblue\neq\tcon$, where $\sblue=\teal$ is possible. 

\begin{defn}\cite[Definition 5.10]{EKo}\label{def:useq}
Given $\sblue,\teal\in S$ with $\sblue\neq\tcon$, the \emph{rotation sequence} associated to $(S,\sblue,\teal)$
is defined as follows. 
Let
\begin{equation}\label{eq.u0u-1}
\textcolor{RoyalBlue}{u_0 = s}
,\quad \textcolor{Magenta}{\textcolor{Magenta}{u_{-1}} =\tcon}  ,  \quad I_0 =S\setminus u_0,\quad I_{-1} =S\setminus u_{-1} 
\end{equation}
and define the elements $u_i \in S$ and subsets $I_i \subset S$, for $i\in \mathbb Z$, by 
\begin{equation} \label{urecurse} u_{i+1} = w_{I_i} u_{i-1} w_{I_i}, \quad I_i = S \setminus u_i. \end{equation}
\end{defn}
See \cite[Section 5.4]{EKo} for further discussions on the rotation sequence and examples (some other examples appear below in Examples~\ref{ex.ab1}-\ref{ex.ab4}), but beware that \cite{EKo} uses the index $(S\setminus\sblue,\sblue,\teal)$ for the rotation sequence we associate to $(S,\sblue,\teal)$. Here we give an alternative characterization of the rotation sequences.

\begin{prop}\label{prop.atominrotation}
Let $u_\bullet$ be the rotation sequence for $(S,\sblue,\teal)$ and let 
\[I_i=S\setminus u_i,\quad L_i=I_i\setminus u_{i-1}.\] 
Then
\begin{equation}\label{eq.at_i}
\at_i=[L_i,I_i,L_{i+1}] = [L_i +u_{i-1}-u_{i+1}]    
\end{equation}
is an atomic coset for each $i\in\mathbb Z$. Conversely, if $u_i\in S$ and $I_i$ are as above, with the cases $i=0,-1$ given by \eqref{eq.u0u-1}, such that \eqref{eq.at_i} is atomic, then $u_\bullet$ is the rotation sequence for $(S,\sblue,\teal)$.
\end{prop}
\begin{proof}
    Both claims are direct consequences of the definitions. 
\end{proof}

It follows from Proposition~\ref{prop.atominrotation} that \begin{equation}\label{eq.tok}
[L_0,I_0,L_1,I_1,L_2,I_2,\cdots,L_k,I_k,L_{k+1}] = [L_0+ u_{-1}- \textcolor{black}{u_1}+\textcolor{RoyalBlue}{u_0} -\textcolor{black}{u_2} +\textcolor{black}{u_1} 
\ \cdots\ 
+u_{k-2} -{u_k}+u_{k-1} - u_{k+1}]     
\end{equation}
is an atomic expression for each $k\geq 0$. The expression \eqref{eq.tok} is reduced if $k$ is small enough, e.g., if $k=0$. This leads to the following definition.

\begin{defn}\label{def.d}
Let $d(S,\sblue,\teal)$ be the maximal $k\geq 0$ such that the atomic expression 
\begin{equation}\label{eq.tok2}
[K+ u_{-1}- \textcolor{black}{u_1}+\textcolor{RoyalBlue}{u_0} -\textcolor{black}{u_2} +\textcolor{black}{u_1} 
\ \cdots\ 
+u_{k-2} -{u_k}+u_{k-1} - u_{k+1}] ,    
\end{equation}
where $K=L_0 =S\setminus\{\sblue,\tcon\}$, 
is reduced.
\end{defn}

\begin{prop}\cite[Proposition 5.14]{EKo}\label{prop.sb}
Let $u_\bullet$ be the rotation sequence for $(S,\sblue,\teal)$. Then, for $d =d(S,\sblue,\teal)$, we have $u_d =\teal$. Moreover, the relation
\begin{equation}\label{sb}
    [I_0, S, I_d]=[I_0+\sblue-\teal] 
    \expr [I_0- \textcolor{black}{u_1}+\textcolor{RoyalBlue}{u_0} -\textcolor{black}{u_2} +\textcolor{black}{u_1} -u_3+\textcolor{black}{u_2} \ \ldots\ -u_{d-1}+u_{d-2} -\textcolor{Green}{u_d}+u_{d-1}] ,
\end{equation}
between two reduced expressions holds. 
The relation of this form is called a \emph{switchback relation}.
\end{prop}
\begin{proof}
In general, a composition of the form $[I\subset J]\circ -$ or $-\circ[I\supset J]$ applied to a reduced expression is a reduced expression (see Definition~\ref{def.p*q}; for example, the $(I,J)$-coset $\pcoset\expr[I\subset J]$ has the maximal element $\ma{\pcoset}=w_J$, whence $(w_Jw_J\inv).\qcoset=e.\qcoset=\qcoset$ gives that $[I\subset J]\circ J_\bullet$ is reduced for any reduced expression $\qcoset\expr J_\bullet$).
In particular, 
\[ [I_0 -\textcolor{black}{u_1}+\textcolor{RoyalBlue}{u_0} -\textcolor{black}{u_2} +\textcolor{black}{u_1} -u_3+\textcolor{black}{u_2} \ \ldots\ -u_{k-1}+u_{k-2} -{u_k}+u_{k-1}]\]
is reduced if and only if \eqref{eq.tok2} is reduced. Therefore, the claims follow directly from \cite[Proposition 5.14]{EKo}.
\end{proof}

An observation for later use:
\begin{rem}\label{rem.redundsb}
    The $(I_0,I_d)$-coset involved in the switchback relation \eqref{sb} is not a core coset but has the left and right redundancies $L_1=I_0\setminus u_1$ and $L_{d}=I_d\setminus u_{d-1}$ respectively. This can be seen in the left hand side of \eqref{sb} by direct computations but also in the right hand side of \eqref{sb} which is a \emph{low road} in the sense of \cite[Section 4.9]{EKo}. 
\end{rem}

We record an additional information on rotation sequences.
\begin{lem}\label{lem.d+1}
Let $u_\bullet$ be the rotation sequence for $(S,\sblue,\teal)$ and $d=d(S,\sblue,\teal)$, as above. Then 
$
{u_{d+1} =\scon}$.
\end{lem}
\begin{proof}
By \cite[Proposition 4.7 and Proposition 5.14]{EKo}, reversing \eqref{sb} yields the switchback relation for the coset $W_{Ls} w_S W_{Lt}\expr [Ls,S,Lt]$. Thus the rotation sequence $v_\bullet$ appearing in this switchback relation is given by $v_i = u_{d-i}$. 
At the same time, the rotation sequence $v_\bullet$ satisfy $\textcolor{Green}{v_0=t}$ and $\textcolor{Orange}{v_{-1}=\scon}$ as given in Definition~\ref{def:useq}. 
In particular, we have $\textcolor{Orange}{u_{d+1}= v_{-1}=\scon}$ as desired.
\end{proof}

\begin{rem}\label{rem.extendingsb}
The first sentence in the proof of Proposition~\ref{prop.sb} says that
\begin{equation}\label{extendsb}
    [L_0+ u- \textcolor{black}{u_1}+\textcolor{RoyalBlue}{u_0} -\textcolor{black}{u_2} +\textcolor{black}{u_1} -u_3+\textcolor{black}{u_2} \ \ldots\ -u_{d-1}+u_{d-2} -\textcolor{Green}{u_d}+u_{d-1} - u'],
\end{equation}
where $u,u'\in S$ are any elements making \eqref{extendsb} well-defined, is reduced. 
Among such $u\in S$, a unique choice which makes 
$[L_0+ u- {u_1}]$
atomic is $u=u_{-1}$.
Similarly, if the last factor $[L_d+u_{d-1} - u']$ is atomic then $u'=u_{d+1}$ must be the case. 
These observations show that \eqref{sb} extends uniquely to the equality between atomic reduced expressions
\begin{equation}\label{atomicsb}
\begin{aligned}
 &[L_0,I_0,S,I_{d},L_{d+1}] \\
 &= [L_0 + \textcolor{Magenta}{u_{-1}} + \sblue - \teal - \textcolor{Orange}{u_{d+1}}]\\
 & \expr [L_0 +\textcolor{Magenta}{u_{-1}} -\textcolor{black}{u_1}+\textcolor{black}{\textcolor{RoyalBlue}{u_0}} -\textcolor{black}{u_2} +\textcolor{black}{u_1} \ \ldots\ -u_{d-1}+u_{d-2} -\textcolor{Green}{u_d}+u_{d-1} - \textcolor{Orange}{u_{d+1}}]
\end{aligned}
\end{equation}
which also agrees with the maximal atomic expression from Definition~\ref{def.d}. 
\end{rem}

Now we obtain a new atomic reduced expression for \eqref{atomicsb} giving rise to a `slightly zoomed-out' relation between expressions of double cosets.

\begin{prop}\label{prop.braid}
For each pair $\sblue,\teal \in S$ such that $\teal \neq \scon$, the relation
    \begin{equation}\label{eq.atombraid}
    \begin{aligned}
     &[K +\textcolor{Magenta}{u_{-1}} -\textcolor{black}{u_1}+\textcolor{RoyalBlue}{u_0} -\textcolor{black}{u_2} +\textcolor{black}{u_1} -u_3+\textcolor{black}{u_2} \ \ldots\ -u_{d-1}+u_{d-2} -\textcolor{Green}{u_d}+u_{d-1} - \textcolor{Orange}{u_{d+1}}]\expr\\
     &[K+\textcolor{RoyalBlue}{\ucon_{d+1}} -\textcolor{black}{\ucon_{d-1}} +\textcolor{Magenta}{\ucon_d} -\textcolor{black}{\ucon_{d-2}} +\textcolor{black}{\ucon_{d-1}} -\ucon_{d-3}+\textcolor{black}{\ucon_{d-2}} \ \ldots\ -\ucon_{1}+\ucon_{2} -\textcolor{Orange}{\ucon_0}+\ucon_{1}- \textcolor{Green}{\ucon_{-1}}]    
    \end{aligned}
    \end{equation}
holds, where $K=S\setminus \{\sblue,\tcon\}$, the rotation sequence $u_\bullet$ is associated to the triple $(S,\sblue,\teal)$, and $d=d(S,\sblue,\teal)$.
In terms of the generating relations in $\SC$, the relation \eqref{eq.atombraid} is the composition
\begin{equation}\label{eq.braidcomp}
    \begin{aligned}
     &[K +\textcolor{Magenta}{u_{-1}} -\textcolor{black}{u_1}+\textcolor{RoyalBlue}{u_0} -\textcolor{black}{u_2} +\textcolor{black}{u_1} -u_3+\textcolor{black}{u_2} \ \ldots\ -u_{d-1}+u_{d-2} -\textcolor{Green}{u_d}+u_{d-1} - \textcolor{Orange}{u_{d+1}}]\\
     &\underset{switchback}{\expr}[K +\textcolor{Magenta}{u_{-1}} +\textcolor{RoyalBlue}{u_0} -\textcolor{Green}{u_d}- \textcolor{Orange}{u_{d+1}}]\\
     &\underset{downdown}{\overset{upup}{\expr}}[K +\textcolor{RoyalBlue}{u_0} +\textcolor{Magenta}{u_{-1}} - \textcolor{Orange}{u_{d+1}}-\textcolor{Green}{u_d}]=[K +\textcolor{RoyalBlue}{\ucon_{d+1}} +\textcolor{Magenta}{\ucon_d} - \textcolor{Orange}{\ucon_0}-\textcolor{Green}{\ucon_{-1}}]\\
     &\underset{switchback}{\expr}[K+\textcolor{RoyalBlue}{\ucon_{d+1}} -\textcolor{black}{\ucon_{d-1}} +\textcolor{Magenta}{\ucon_d} -\textcolor{black}{\ucon_{d-2}} +\textcolor{black}{\ucon_{d-1}} -\ucon_{d-3}+\textcolor{black}{\ucon_{d-2}} \ \ldots\ -\ucon_{1}+\ucon_{2} -\textcolor{Orange}{\ucon_0}+\ucon_{1}- \textcolor{Green}{\ucon_{-1}}]    
    \end{aligned}
\end{equation}
where the middle relations is the composition of two commuting relations upup and downdown.
\end{prop}
\begin{proof}
Note that $K=L_0$ and the first relation in \eqref{eq.braidcomp} is \eqref{atomicsb} written reversed.
The next relation (upup and downdown) in \eqref{eq.braidcomp} is clear and, since $\textcolor{RoyalBlue}{u_0}=\sblue =\textcolor{RoyalBlue}{\ucon_{d+1}}$, by Lemma~\ref{lem.d+1}, and $\textcolor{Green}{u_d} =\teal = \textcolor{Green}{\ucon_{-1}}$, by Proposition~\ref{prop.sb}, we arrive at the expression
\begin{equation}\label{eqafteruudd}
[K + \textcolor{RoyalBlue}{\ucon_{d+1}}+ \textcolor{Magenta}{u_{-1}} - \textcolor{Orange}{u_{d+1}} - \textcolor{Green}{\ucon_{-1}}].
\end{equation}
The expression \eqref{eqafteruudd} contains the non-atomic expression
\begin{equation}\label{tconscon}
    [K\sblue+\textcolor{Magenta}{u_{-1}} - \textcolor{Orange}{u_{d+1}}] = [K\sblue+ \tcon - \scon]
\end{equation}
where the equality between the elements is given again by Lemma~\ref{lem.d+1}.
The switchback relation for \eqref{tconscon} is \rev{obtained from \eqref{sb} by applying the automorphism $u\mapsto \ucon$ and reversing the expression (see Proposition~\ref{prop.inv}):}
\begin{equation}\label{sb2}
\begin{aligned}
    [Ks
    , S, S\setminus \con{s}] &\expr [Ks
    - \textcolor{black}{\ucon_{d-1}} +\tcon -\textcolor{black}{\ucon_{d-2}} +\textcolor{black}{\ucon_{d-1}} -\ucon_{d-3}+\textcolor{black}{\ucon_{d-2}} \ \ldots\ -\ucon_{1}+\ucon_{2} -\scon+\ucon_{1}]\\
    &=[Ks
    - \textcolor{black}{\ucon_{d-1}} +\textcolor{Magenta}{\ucon_d} -\textcolor{black}{\ucon_{d-2}} +\textcolor{black}{\ucon_{d-1}} -\ucon_{d-3}+\textcolor{black}{\ucon_{d-2}} \ \ldots\ -\ucon_{1}+\ucon_{2} -\textcolor{Orange}{\ucon_0}+\ucon_{1}].
\end{aligned}
\end{equation}
Putting \eqref{sb2} back into \eqref{eqafteruudd}, we obtain the rest of \eqref{eq.atombraid}. 
This completes the proof.
\end{proof}

\begin{defn}\label{def.braid}
The \emph{atomic braid relation} associated to $(S,s,t)$, for $t\neq\con{s}$, is the relation  \eqref{eq.atombraid}.
\end{defn}

\rev{In other words,} letting 
\[\at= W_K w_{K\bar{t}} W_{K\bar{t}\setminus u_1}\expr [K+\tcon - u_1],\qquad 
\at'\expr
[Ks\setminus \bar{u}_{d-1}+\tcon - \bar{u}_{d-2}]\] and 
\[\atb = W_Kw_{Ks}W_{Ks\setminus \bar{u}_{d-1}}\expr[K+\sblue -  \bar{u}_{d-1}],\qquad 
\atb'\expr 
[K\bar{t}\setminus u_1 + \sblue-u_2],\]
the {atomic braid relation} associated to $(S,s,t)$ is 
\begin{equation}\label{eq.atomicab'}
    \underbrace{\at*\atb'*\cdots}_{m_{\at,\atb}} = \underbrace{\atb*\at'*\cdots}_{m_{\atb,\at}}
\end{equation}
whose left (resp., right) hand side is the (unique) longest reduced atomic product in $\SC$ starting with $\at*\atb'$ (resp., $\atb*\at'$). 
%
%
\rev{We have $m_{\at,\atb}=m_{\atb,\at}=d(S,s,t)+1$ by Definition~\ref{def.d}.

Here we think of $\mathtt{a}'$ as a twist of $\at$ since it is the unique atom of the form $[L'+\tcon - v']$ such that $\atb*\at'$ is well-defined. 
The ``$\cdots$'' parts in \eqref{eq.atomicab'} are described using a similar idea. 
Given an atom $\mathtt c \expr [L+u-\textcolor{violet}{v}]$ we denote by $\hat{\mathtt{c}}$ any atom of the form $[L'+\textcolor{violet}{v}-v']$. It makse sense to use $\hat{\mathtt{c}}$ in a composition since there is unique $L'$, depending on $p\in \SC$, such that $p*\hat{\mathtt{c}}$ is well-defined (and unique $v'$ such that $\hat{\mathtt{c}}$ is an atom). 
Then we can write \eqref{eq.atomicab'} more concretely as
\begin{equation}\label{eq.withhats}
    \underbrace{\at*\atb'*\hat\at *\hat{\atb'}*\hat{\hat{\at}}*\cdots}_{m_{\at,\atb}} = \underbrace{\atb*\at'*\hat\atb*\hat{\at'}*\hat{\hat\atb}*\cdots}_{m_{\atb,\at}}.
\end{equation}

Thus the atomic braid relation \eqref{eq.atomicab'} is of the same form as the regular braid relation, say between two simple generators $a$ and $b$ with $m_{ab} = m_{\at\atb}$, 
except that the composition does not alternate $\at$ and $\atb$ themselves but their appropriate twists (note that $\mathtt c'=\mathtt c=\hat{\mathtt c}$ if they are regular). See Examples~\ref{ex.ab2},~\ref{ex.abE8} for some nontrivial twists.}

\begin{ex}\label{ex.ab1}
    If $S=\{s,t\}$, 
    then \eqref{eq.atombraid} is exactly the regular braid relation
    \[ [\underbrace{st\cdots}_{m_{st}}] \expr \underbrace{[\emptyset,\{s\},\emptyset,\{t\},\emptyset,\cdots ]}_{2m_{st}} \expr \underbrace{[\emptyset,\{t\},\emptyset,\{s\},\emptyset, \cdots ]}_{2m_{st}}\expr[\underbrace{ts\cdots}_{m_{st}}].\]
\end{ex}

\begin{ex}\label{ex.commutingatoms}
A switchback relation with $d=1$ is of the form $[+s-t]=[-t+s]$ (this happens when $s,t$ belong to different irreducible compoments of $S$), so the associated atomic braid relation is the $m=2$ braid relation given by the composition
\begin{equation}
    [K+u-s+t-v]\underset{switchback}{\expr} [K+u+t-s-v]\underset{downdown}{\overset{upup}{\expr}} [K+t+u-v-s]\underset{switchback}{\expr} [K+t-v+u-s].
\end{equation}
\end{ex}

\begin{ex}\label{ex.ab2}
If $S$ is of type $A$, a switchback relation has either $d=1$ or $d=2$. So the atomic braid relations in this case are either described in Example~\ref{ex.commutingatoms} or of the form resembling the $m=3$ braid relation\rev{, for example, 
\[[K+s_4-s_9+s_3-s_6+s_9-s_7]\expr [K+s_3-s_1+s_4-s_7+s_1-s_6]\]
where $K=S\setminus\{s_3,s_4\}$ and $S=\{s_1,\cdots,s_9\}$ is of type $A_9$.}
\end{ex}

In a general (finite) type, the number $d+1$ may be strictly bigger than any $m_{st}$ for $s,t\in S$.
\begin{ex}\label{ex.ab4}
 Here is an example in type $D_4$ with $d+1 = 4$:
\[[\{s_1,s_3\}+s_2 -s_2+s_4 -s_4+s_2 -s_2+s_4 -s_4]\expr
[\{s_1,s_3\}+s_4 -s_4+s_2 -s_2+s_4 -s_4+s_2 -s_2],\]
where $s_2$ is the special element in $S=\{s_1,\cdots,s_4\}$ corresponding to the hub of the $D_4$ Dynkin diagram. 
\end{ex}
\begin{ex}\label{ex.abE8}
Here is an example in type $E_8$ with $d+1 = 8$:
\begin{equation*}
    \begin{aligned}
    &[K +s_1 -s_3 +s_4-s_7+s_3-s_5+s_7-s_7+s_5-s_3+s_7-s_4+s_3-s_1+s_4-s_4]\\
    &\expr
    [K +s_4-s_4+s_1-s_3+s_4-s_7+s_3-s_5+s_7-s_7+s_5-s_3+s_7-s_4+s_3-s_1]    
    \end{aligned}
\end{equation*}
where $K=\{s_2, s_3, s_5, s_6, s_7, s_8\}$ and it is an exercise for the reader to find how we index $S=\{s_1,\cdots,s_8\}$.
\end{ex}

\begin{rem}\label{rem.param}
The atomic braid relation associated to $(S,s,t)$ is the same as that associated to $(S,\con{t},\con{s})$.
It follows that the atomic braid relations are parametrized by the (unordered) subsets 
\[\{\{s,t\}\subset S\ |\ t\neq\con{s}\},\] while the switchback relations are parametrized by the ordered pairs 
\[\{(s,t)\in S\times S\ |\ t\neq\con{s}\}.\] 
The full list of the switchback relations, for irreducible $S$, is given in \cite[Section 6]{EKo}, and \cite[Section 5.6]{EKo} shows that the switchback relation for $(S,s,t)$ is trivial, i.e., of the form $[L+s-t]\expr[L+t-s]$ when $s,t$ belong to different connected component of $S$ and has the same form as the switchback relation for $(S^0,s,t)$ if $s,t\in S^0$ for a connected component $S^0$. 
This gives a classification of the atomic braid relations.
\end{rem}

An implicit point of Proposition~\ref{prop.braid} is another characterization of atomic braid relations, which we make explicit in the following two propositions.
\begin{prop}\label{prop.rexgraph}
Let $s,t\in S$ be such that $t\neq \con{s}$ and let 
\[p\expr [K+\tcon + \sblue - \teal - \scon]\rev{=:E_2}.\]
Then the reduced expression graph of $\pcoset$ is of the form
\begin{equation}\label{atomicbraidgraph}
\begin{tikzcd}
        &[3em]&[3em]  \arrow[dl,dash,swap,"downdown"] E_3 \arrow [dr,dash,"upup"]  &[3em]&[3em]   \\
  E_1 \arrow[r,dash,"switchback"] & E_2   \arrow [dr,dash,swap,"upup"]    &  &  \arrow[dl,dash,"downdown"] E_5  &  E_6 \arrow[l,dash,"switchback"]\\
   & & E_4 & & ,
\end{tikzcd}    
\end{equation}
\rev{
where $E_1,E_2,E_5,E_6$ are the expressions in \eqref{eq.braidcomp}, in the same order, and 
\[E_3=[K+\tcon + \sblue - \scon -\teal],\qquad E_4=[K+ \sblue +\tcon -\teal - \scon].\]} 
\end{prop}

\begin{proof}
By Proposition~\ref{prop.braid} and \rev{Theorem~\ref{thm.mats}}, it is enough to show that no braid relation in $\SC$ other than what appears \rev{ in \eqref{atomicbraidgraph} is applicable to each expression in \eqref{atomicbraidgraph}.} 

First consider \rev{$E_1$, namely} the first expression in \eqref{eq.braidcomp}.
Since it has alternating signs, no upup or downdown relation is applicable. Also, by Proposition~\ref{prop.atominrotation} and Proposition~\ref{prop.atom}, no switchback relation is applicable to any $[u_{i-1}-u_{i+1}]$. 
Therefore it remains to note that the only (consecutive) subexpression where a switchback relation is applicable is 
$[-u_1\ \cdots + u_{d-1}]$, since such a subexpression cannot start earlier or later for sign reasons and cannot be a subexpression of $[-u_1\ \cdots + u_{d-1}]$ by maximality of $d$ in switchback relations.

For the second expression in \eqref{eq.braidcomp}, it is clear that there is one place that upup is applicable, one place that downdown is applicable, and one place that switchback is applicable. 

For $E_3,E_4$, 
no switchback is applicable to the atomic subexpressions $[\sblue - \scon]$ and $[\tcon -\teal]$, respectively, and thus upup and downdown are the only applicable relations. 

The rest follows by symmetry.
\end{proof}

\begin{prop}\label{prop.suds}
Let $\pcoset$ be a double coset of the form $\pcoset\expr[K+s+t-v-u]$.
\begin{enumerate}
\item\label{graphtobraid} If $\pcoset$ has a rex graph of the form 
\eqref{atomicbraidgraph}, where $E_1,\cdots ,E_6$ are the reduced expressions of $\pcoset$, 
then the composition $E_1\expr E_6$ is an atomic braid relation.
\item If $t\neq w_{Kst} v w_{Kst}$ and if applying the switchback relation for $[+t-v]$ in $[K+s+t-v-u]$ yields an atomic expression, then the rex graph of $\pcoset$ is of the form \eqref{atomicbraidgraph} (which gives an atomic braid relation by Part~\eqref{graphtobraid}). 
\end{enumerate}
\end{prop}
\begin{proof}
We set $ S = Kst$ and, by renaming, assume that $t\neq \con{v}$ (if $t=\con{v}$ then swap $v$ and $u$ and apply the downdown relation). 

Let us prove the first statement. 
Let $p\expr [K+s+t-v-u]$ has a rex graph \eqref{atomicbraidgraph}.
Switchback, upup, and downdown relations are applicable to $[K+s+t-v-u]$ so we have to have  $[K+s+t-v-u]=E_2$, up to the symmetry of \eqref{atomicbraidgraph}.
Then $E_1$ is of the form $[K+s -u_1+u_0 -u_2 + u_1 \ \cdots - u_d + u_{d-1} -u]$ with the rotation sequence $u_\bullet$. Since $E_1$ has no other edge in the graph, it involves no other switchback relations, in particular, $[+s-u_1]$ and $[+u_{d-1} - u]$ are atomic.
The latter implies $s=u_{-1}$ and $u = u_{d+1}$ (see remark~\ref{rem.extendingsb}) and the claim is proved. 

Now suppose the right hand side of 
\[[K+s+t-v-u]\underset{switchback}{\expr}[K+s -u_1+u_0 -u_2 + u_1 \ \cdots - u_d + u_{d-1} -u]\] is atomic. Then we have $s=u_{-1}$ and $u = u_{d+1}$, and Proposition~\ref{prop.sb} and Proposition~\ref{prop.rexgraph} gives the second statement.
\end{proof}

\section{The atomic Matsumoto theorem}\label{s.mats}

In this section we let $(W,S)$ be an arbitrary Coxeter system and $I,J\subset S$ be finitary subsets.

\begin{thm}\label{mats}
Given a core $(I,J)$-coset $\pcoset$, any two atomic reduced expressions of $\pcoset$ are related by a composition of atomic braid relations (associated to $(M,s,t)$ for various $M\subset S$ finitary and $s,t\in M$ with $t\neq w_Msw_M$).
\end{thm}
\begin{proof}
Given two atomic expressions $L_\bullet$ and $K_\bullet$ of a core coset, we write $L_\bullet \sim K_\bullet$ if $L_\bullet$ and $K_\bullet$ are related by a composition of atomic braid relations.
Recall that the \emph{width} of an expression $[L_0,L_1,\cdots, L_m]$ is the number $m$. 

We proceed by induction on the minimal width of atomic reduced expressions of $\pcoset$. The base case $\pcoset\expr [I]$ has a unique reduced expression, hence the claim of the theorem is trivially true. 

Suppose that the claim is true for every core coset with an atomic reduced expression of width $<m$.
Let 
\begin{equation}\label{eq.p=sss}
p\expr L_\bullet = [L_0,L_1,\cdots, L_m] =[I+s-s' +  \ldots - s'']    
\end{equation}
and
\begin{equation}\label{eq.p=ttt}
p\expr K_\bullet = [K_0,K_1,\cdots, K_l] =[I+t-t' + \ldots - t'']    
\end{equation}
be two atomic reduced expressions (where $l$ is any positive integer). 

If $s=t$, then we also have 
$s'= w_{Is} s w_{Is}=t'$ since $L_\bullet,K_\bullet$ are atomic, thus $L_i=K_i$ for $i=0,1,2$.
\rev{Since $L_\bullet,K_\bullet$ are reduced, 
\[\ma{p} = w_{L_1}w_{L_2}\inv \ma{[L_2,\cdots, L_m]} = w_{K_1}w_{K_2}\inv\ma{[K_2,\cdots, K_l]}\]
holds by \eqref{eq.pmarex}, where $\ma{M_\bullet}$ denotes the maximal element in the coset $M_\bullet$ expresses.
We deduce
\[[L_2,\cdots, L_m]\expr [K_2,\cdots, K_l]\]
whose sides are two atomic reduced expressions covered by the induction hypothesis. 
Therefore,} the claim
$L_\bullet\sim K_\bullet$ follows from $[L_2,\cdots, L_m]\sim [K_2,\cdots, K_l]$. 

It is thus enough to consider the case $s\neq t$. By \cite[Corollary 4.20]{EKo}
, we have $s,t\in \leftdes(\ma{p})$ and by \cite[Proposition 4.21]{EKo} there exist reduced expressions (called high roads) of the form
\[p\expr[I,Is,Ist]\circ J_\bullet \expr [I,It,Ist]\circ J_\bullet,\]
where $J_\bullet$ expresses some $(Ist,J)$-coset $\qcoset$.
Then we have 
\[\mi{q}\inv J \mi{q}= \mi{q}\inv \mi{p}I \mi{p}\inv \mi{q}\subset W_{Ist} I W_{Ist} = W_{Ist}\]
where we use that $p$ is core in the first equality and $\mi{q}\in W_{Ist}\mi{p}$ in the inclusion.
By Kilmoyer's theorem \cite[Lemma 2.14]{EKo}, we obtain $ Ist\supset \mi{q}\inv J\mi{q}$ and thus
\[\leftred(q) = Ist\cap \mi{q}\inv J\mi{q}= \mi{q}\inv J\mi{q}.\] 
Note also that $|\leftred(q)| = |J|=|I| = |Ist|-2$. Naming  $\{u,v\} =  Ist\setminus \leftred(q)$, we have by Proposition~\ref{prop.lowroad}
\begin{equation*}
    q\expr [Ist - u - v]\circ q^{\core} 
\end{equation*}
from which it follows
\begin{equation}\label{eq.pqcore}
    p\expr  [I+s+t-v-u]\circ q^{\core}
\end{equation}
We may and do assume that $u\neq w_{Ist}\inv s w_{Ist}$ and $v\neq w_{Ist}\inv t w_{Ist}$, since we can swap $u$ and $v$ if an equality holds.
 
Applying the switchback relation for $(Ist,\textcolor{RoyalBlue}{t}, \textcolor{Green}{v})$ 
to \eqref{eq.pqcore}, we get
\begin{equation}\label{tvsb}
     p\expr [I + s {+ \textcolor{RoyalBlue}{t} - \textcolor{Green}{v}} - u ] \circ q^{\core}  \expr [I + s {-u_1 + \textcolor{RoyalBlue}{u_0} - \ldots - \textcolor{Green}{u_d} + u_{d-1}} -u ]\circ q^{\core} .
\end{equation}
By Lemma~\ref{lem.claim} below, the expression $[I + s {-u_1 + \textcolor{RoyalBlue}{u_0} - \ldots -\textcolor{Green}{u_d} + u_{d-1}} -u ]$ is atomic. Since $[I+s-s']$ is also atomic (recall the atomic expression \eqref{eq.p=sss}) we have $u_1=s'$. \rev{In the rest of the proof we let 
$u':=u_{d-1}$
to make the equations more readable.}

Now, taking an atomic reduced expression $M_\bullet\expr q^{\core}$, we obtain the atomic expression 
\[ [L_0,L_1,L_2\cdots, L_m] =[I+s-s' +  \ldots - s''] \expr [I + s {-s' + \textcolor{RoyalBlue}{t} - \ldots - \textcolor{Green}{v} + u'} -u ]\circ M_\bullet\] and  
\begin{equation}\label{eq.subtractss'}
   [L_2,\cdots,L_m]\expr [(Is\setminus s')+\textcolor{RoyalBlue}{t} - \ldots - \textcolor{Green}{v} + u' -u ]\circ M_\bullet 
\end{equation}
where the right hand side, being a subexpression of such, is an atomic reduced expression.
Then induction hypothesis gives $[L_2,\cdots, L_m]\sim  [(Is\setminus s')+\textcolor{RoyalBlue}{t} - \ldots - \textcolor{Green}{v}+ u' -u ]\circ M_\bullet$ , from which it follows 
\begin{equation}\label{eq.indstep}
    L_\bullet\sim [I + s {-s' + \textcolor{RoyalBlue}{t} - \ldots -\textcolor{Green}{v} + u'} -u ]\circ M_\bullet.
\end{equation}

Since the same arguments establishes (an analogue of) \eqref{eq.indstep} for $K_\bullet$, we 
have
\begin{equation}
\begin{aligned}
    L_\bullet &\overset{\eqref{eq.indstep}\text{ for $L_\bullet$}}{\sim}  [I + s \underbrace{-s' + t - \ldots -v + u'} -u ]\circ M_\bullet \\
    &\overset{\textcolor{black}{switchback}}{\expr} 
    [I + s \overbrace{+ t - v} - u ] \circ M_\bullet  \\
    &\underset{downdown}{\overset{upup}{\expr}}  [I + t \underbrace{+ s - u} - v ]\circ M_\bullet \\
    &\overset{\textcolor{black}{switchback}}{\expr} [I +t \overbrace{-t'+s - \ldots -u + v'} - v ]\circ M_\bullet \\
    &\underset{\eqref{eq.indstep}\text{ for $K_\bullet$}}{\sim} K_\bullet,
\end{aligned} 
\end{equation}
where, by Proposition~\ref{prop.suds}, the middle three relations compose to an atomic braid relation. The proof is complete.
\end{proof}

\begin{lem}\label{lem.claim}
    If $u_\bullet$ is a rotation sequence and 
    \[\pcoset\expr [I + s {-u_1 + u_0 - \ldots - u_d + u_{d-1}} -u ]\circ \rcoset\] is a core coset, then 
    $[I + s {-u_1 + u_0 - \ldots - u_d + u_{d-1}} -u ]$ is an atomic expression.
\end{lem}
\begin{proof}
By Proposition~\ref{prop.atominrotation}, it is enough to show that $[I+s-u_1]$ is atomic and that $[+u_{d-1}-u]$ is atomic.

Let us prove that $[I+s-u_1]$ is atomic. Suppose not. Then, by Proposition~\ref{prop.sb}, a switchback relation applies to $[I+s-u_1]$ to give a reduced expression of $\pcoset$ starting with some $[I \supset I']$. 
Since the redundancy decreases in the first step of the latter expression (see also remark~\ref{rem.redundsb}), this contradicts  $\pcoset$ being a core coset (see e.g., \cite[Lemma 4.25]{EKo}).

Similarly, if the subexpression $[+u_{d-1}-u]$ is not atomic, then remark~\ref{rem.redundsb} applies here to result in a contradictory size of the redundancy for $p$.
\end{proof}

\begin{rem}
    Theorem~\ref{mats} does not say that the atomic reduced expressions are in correspondence with the reduced expressions in some Coxeter group. 
    Although this is true in type $A$ and type $B$ (as is proved in \cite[Theorem 3.11, Theorem 4.7]{KELP3}), new structures appear in other types. See Section~\ref{s.highrank}.
\end{rem}

\begin{rem}
    Theorem~\ref{mats} does not say that any relation between two atomic reduced expressions in $\SC$ is a composition of atomic braid relations. 
    For a core coset $\pcoset$ and its atomic reduced expressions, it seems more reasonable to consider the rex graph whose edges are the atomic braid relations, rather than the full subgraph of the rex graph for $\pcoset$ in $\SC$.
\end{rem}

\begin{cor}\label{cor.atomlen}
    All atomic reduced expressions of a given core coset are of the same width. 
\end{cor}
\begin{proof}
    This follows from Theorem~\ref{mats} since every atomic braid relation relates expressions of the same width.
\end{proof}

The width of an atomic expression of a core coset $\pcoset$ is an even number $2l$, where $l$ is the number of atomic cosets that are composed in the expression. 
Corollary~\ref{cor.atomlen} says that the following is well-defined.

\begin{defn}
    The \emph{atomic length} of a core coset $\pcoset$ is half of the width of an atomic expression of $\pcoset$. 
\end{defn}

\begin{rem}
One may define the atomic length for all double cosets by setting the atomic length of $\pcoset\in\SC$ to be the atomic length of $\pcoset^{\core}$, but this definition lacks desirable properties. 
For example, the atomic length of $(p\circ q)^{\core}$ need not agree with the sum of the atomic lengths of $\pcoset^{\core}$ and of $\qcoset^{\core}$. 
Also, the atomic length is not compatible with the Bruhat order in the sense that there exist $(I,J)$-cosets $\pcoset\leq q$ where the atomic length of $\pcoset^{\core}$ is strictly bigger than that of $\qcoset^{\core}$, even for $\qcoset=q^{\core}$ atomic.
\end{rem}

\section{Non-braid relations}\label{s.nonbraid}


We continue letting $(W,S)$ be an arbitrary Coxeter system.
Let us introduce two classes of non-braid relations between atomic expressions.

\begin{defn}
Associated to an atomic expression in $\SC$ of the form
$[I+s-s]$, we define the \emph{atomic quadratic relation} as 
\begin{equation}\label{quadratic}
    [I+s\underbrace{-s+s}-s] \underset{*-quadratic}{\expr} [I+s-s].
\end{equation}
Associated to an atomic expression in $\SC$ of the form $[I+s-t]$ with $s\neq t$, we define the \emph{atomic cubic relation} as the composition (of two commuting relations)
\begin{equation}\label{cubic}
    [I+s\underbrace{-t+t}\overbrace{-s+s}-t] \overset{*-quadratic}{\underset{*-quadratic}{\expr}} [I+s-t].
\end{equation}
\end{defn}

The relation \eqref{cubic} still holds when $s=t$ but is generated by \eqref{quadratic}, so we exclude such cases.

\begin{rem}\label{rem.nonbraidrels}
The left hand side of the atomic quadratic relation \eqref{quadratic} and the atomic cubic relation \eqref{cubic} are indeed atomic expressions. The only nontrivial part is $[(Is\setminus t) +t-s]$ in \eqref{cubic} being atomic, but this follows from Remark~\ref{rem.inverseatom} since the latter is the \rev{reverse} of the atomic expression $[I+s-t]$.
\end{rem}

The atomic quadratic relation (or the absence of other quadratic relations) is motivated by the following statement from \cite{KELP3} (see \cite[Lemma 2.10, Proposition 2.20]{KELP3}). We include a proof for convenience. 

\begin{prop}\label{atomatom}
Let $\at\expr [I+s-t]$ be an atomic $(I,J)$-expression (where $J=Is\setminus t$) and $\atb\expr [J+u-v]$ be an atomic $(J,K)$-expression.
\begin{enumerate}
    \item\label{atomatom1} If $\pcoset=\at*\atb\expr [I+s-t]\circ[J+u-v]=[I+s-t+u-v] $ is reduced, then $\pcoset$ is a core coset.
    \item\label{atomatom2} If the composition $\pcoset=\at*\atb$ is not reduced, then $\pcoset$ is core if and only if $I=J=K$ and $s=t=u=v$, in which case $\pcoset\expr[I+s-s]$.
\end{enumerate}
\end{prop}
\begin{proof}
     The first part is a special case of Lemma~\ref{corecore}. The `if' claim follows from the $*$-quadratic relation $[Is-s+s]\expr[Is]$ which also gives  $p\expr[I+s-s]$ in this case.

     Now suppose $p=\at*\atb\expr [I+s-t+u-v]$ is not reduced while $p$ is core. Then $q\expr [I+s-t+u]$ is not reduced, that is, $\mi{q}<\mi{\at}$ or $\leftred(q) \supsetneq \leftred(\at)$. The latter is not the case since $\at$ is core. 
     Since $\mi{\at} = \mi{q}.x$ for some $x\in W_{Ju}$ and since the only right descent of  $\mi{\at}$ is $t$, we have $t=x\in Ju$, thus $t=u$. It follows $s=v$ and $I=K$.
     Finally, the $*$-quadratic relation gives $p=[I+s-s]$ implying that $p$ is an atom and thus $p=\at$. The other equalities follows.
\end{proof}

Recall that by an {atomic expression} we mean a composition of (the unique reduced expressions of) atomic cosets, which is not necessarily reduced or core. We introduce an auxiliary definition useful for this section. 
\begin{defn}\label{def.adm}
An atomic expression $[I_0,\cdots, I_{2\ell}]$ is said to be \emph{admissible} if each $[I_0\cdots, I_{2m}]$ for $m\leq \ell$ expresses a core coset.
\end{defn}

\begin{rem}\label{rem.rexadm}
    A reduced atomic expression $[I_0,\cdots, I_{2\ell}]$ is automatically admissible since the redundancies for all $p_i\expr [I_0,\cdots,I_i]$ are the same (see Lemma~\ref{lem.redconst}).
    \rev{ In this case, every atomic subexpression \([I_{2k}, ..., I_{2m}]\) expresses a core coset. In general we do not know whether an atomic subexpression of an admissible expression is also an expression of a core coset.}
\end{rem}

We then have the following partial analogue of the Coxeter presentation.

\begin{prop}\label{prop.reducing}
A non-reduced admissible atomic expression is related to a reduced atomic expression by a composition of atomic braid \rrev{and} atomic quadratic relations\rrev{. 
Moreover, atomic quadratic 
relations may be applied only from the left hand side to the right hand side of \eqref{quadratic}.
}
\end{prop}

Our proof of Proposition~\ref{prop.reducing} uses the following analogue of the (weak) exchange property for Coxeter groups.
\begin{lem}\label{lem.exchange}
    Let $\qcoset$ be a core $(I,J)$-coset and suppose $r = q\circ[J+s]$ is a (well-defined) non-reduced composition.
    \begin{enumerate}
        \item We have $\mi{q}> \mi{r}$.
        \item The core coset $\qcoset$ has an atomic reduced expression ending in $[-s]$.
    \end{enumerate}   
\end{lem}
\begin{proof}
Since $\qcoset\subset r$ we have $\mi{q}\geq \mi{r}$.
    Suppose $\mi{q}=\mi{r}$. Then 
     \[\rightred(r)\underset{def}{=}Js\cap \mi{r}\inv I \mi{r} \underset{\mi{q}=\mi{r}}{=} Js\cap \mi{q}\inv I \mi{q} = Js\cap J = \rightred(q),\]
     where we use that $\qcoset$ is a core coset in the last two equalities. It follows that  $q\circ[J+s]$ is reduced, contrary to the assumption. This proves the first claim.

     The condition $\mi{q}> \mi{r}$ implies $t\in\rightdes(\mi{q})$ for some $t\in Js$. But since $\rightdes(\mi{q})\cap J =\emptyset$, we have $t=s$. 
     Then 
\[s\in \rightdes(\mi{q})\subset \rightdes (w_I.\mi{q})=\rightdes(\ma{q})\] 
since $\qcoset$ is core. Now Proposition~\ref{prop.arexfinns} guarantees an atomic reduced expression of $\qcoset$ ending with $[-s]$. 
\end{proof}

\begin{proof}[Proof of Proposition~\ref{prop.reducing}]
We proceed by induction on the atomic length, which agrees with the induction on the width. Our induction begins with two base cases, namely the reduced expressions $[I]$ and $[I+s-t]$, where the claim is trivial. 

Let $m\geq 1$ and suppose the claim is true for any non-reduced admissible atomic expression of atomic length $\leq m$ and let 
\[p \expr \at_0*\cdots *\at_m \expr I_\bullet = [I_0,\cdots , I_{2m+2}]= [I+s_0-t_0+s_1-t_1\ \cdots +s_m-t_m]\] be a non-reduced admissible atomic expression of atomic length $m+1$, where each $\at_i$ is an atomic coset. We work with the `expression' $\at_0*\cdots *\at_m$ instead of the expression $I_\bullet$, identifying each $\at_i$ with its unique reduced expression $[I_{2i}+s_i-t_i]$.

The proper subexpression $\qcoset \expr \at_0*\cdots *\at_{m-1}$ is \rev{admissible} since $I_\bullet$ is an admissible atomic expression.
Applying the induction hypothesis, we assume that $\qcoset\expr p_0*\cdots *\pcoset_{m-1}= [I_0,\cdots, I_{2m}]$ is reduced.
Since $I_\bullet$ is not reduced, we have that $r\expr q\circ [I_{2m}+s_m]$ is not reduced. 

By Lemma~\ref{lem.exchange}, there is an atomic reduced exprssion $\qcoset\expr M_\bullet=[M_0,\cdots ,M_{2m}]$ so that \begin{equation}\label{eq1}
I_\bullet \sim M_\bullet\circ [I_{2m}+s_m - t_m] = [M_0,\cdots , M_{2m-2}]\circ [I_{2m-2}+t -s_m]\circ [I_{2m}+ s_m - t_m]     
\end{equation}
holds for some $t\in S$. Here `$\sim$' denotes a composition of atomic braid relations and follows by Theorem~\ref{mats}. The latter also justifies the width of $M_\bullet$ being $2m$.
Since $[I_{2m-2}+t -s_m]=[M_{2m-2},M_{2m-1},M_{2m}]$ is atomic, we have in fact $t = w_{I_{2m}s_m} s_m w_{I_{2m}s_m}$. Since $[I_{2m}+s_m-t_m]$ is also atomic, we have $t=t_m$. 

\rrev{If $s_m=t_m$, then} the atomic quadratic relation \eqref{quadratic} applies to the last part of \eqref{eq1}, in the correct direction. The resulting atomic expression has the atomic length $m$. This establishes the induction step in the case $s_m=t_m$. 
\rrev{To complete the induction step, we show in the next paragraph that the other case $s_m\neq t_m$ does not occur.}

\rrev{S}uppose $s_m\neq t_m$. Then  
\[[M_{2m-2}+t_m-s_m+s_m-t_m]\expr[M_{2m-2}+t_m-t_m]\]
neither is atomic nor expresses a core coset. In particular, we cannot have $\pcoset\expr [M_{2m-2}+t_m-s_m+s_m-t_m]$, i.e., we have $m\geq 2$.
Now consider the non-atomic composition
\[[M_0,\cdots, M_{2m-2}]\circ [+t_m-t_m]\] which is not reduced by Lemma~\ref{corecore}.
Since $[M_0,\cdots, M_{2m-2}]$ is reduced, it is the step `$+t_m$' that is not reduced, namely, we are again in the situation in Lemma~\ref{lem.exchange}. 
Lemma~\ref{lem.exchange} gives an atomic reduced expression $\rrev{[N_0,\cdots,N_{2m-2}]} \expr [M_0,\cdots, M_{2m-2}]$  ending in `$-t_m$'. 
We see that \eqref{eq1} continues to
\begin{equation}\label{eq.long}
\begin{split}
    I_\bullet&\sim [M_0,\cdots , 
    M_{2m-3}, M_{2m-2}] \circ [+t_m-s_m+s_m-t_m]\\
    &\sim [N_0,\cdots , 
    N_{2m-3}, N_{2m-2}] \circ [+t_m-s_m+s_m-t_m] \\
    &= [N_0,\cdots, N_{2m-\rrev{3}}]\circ [ 
    -t_m+t_m-s_m+s_m-t_m]\rrev{.}
\end{split}
    \end{equation}
The second equivalence in \eqref{eq.long} follows from Theorem~\ref{mats} since  $[M_0,\cdots, M_{2m-2}]\expr N_\bullet$ involves reduced atomic expressions.
\rrev{But \eqref{eq.long} implies 
\[M_\bullet\sim [N_0,\cdots,N_{2m-3}]\circ[\underbrace{- t_m+t_m}-s_m],\]
which contradicts $M_\bullet$ being reduced.}

\end{proof}

\begin{cor}\label{presentation}
All admissible atomic expressions of a (core) coset are related by atomic braid relations\rrev{ and }
atomic quadratic relations\rrev{.}
\end{cor}
\begin{proof}
    This follows from Proposition~\ref{prop.reducing}, Remark~\ref{rem.rexadm}, and Theorem~\ref{mats}.
\end{proof}

\begin{rem}
    We may have a stronger analogue of the Coxeter presentation than Corollary~\ref{presentation} by including non-admissible atomic expressions. \rrev{The analogue should involve not only the atomic braid and quadratic relations but also the atomic cubic relations \eqref{cubic}.} \rev{A possible proof is } by induction similar to that of Proposition~\ref{prop.reducing}, but we need to consider the possibility that \rev{the coset $\qcoset$ used in the proof of Proposition~\ref{prop.reducing}} is not a core coset. The latter leads us to remark~\ref{rem.allcoset}.
\end{rem}

\begin{rem}\label{rem.allcoset}
We may have a different analogue of the Coxeter presentation than Corollary~\ref{presentation} by including atomic expressions of non-core cosets, namely, all cosets generated by the atomic cosets.
There a class of generating non-braid relations would be
\begin{equation}\label{eq.newquad}
    [I+s-t+t-s]\expr [I+s-s]
\end{equation}
arising from the composition $\at*\at\inv$ for each atomic coset $\at\expr[I+s-t]$\rrev{, replacing the old quadratic relations \eqref{quadratic}, and the cubic relations~\eqref{cubic} would also be necessary}.
But we may need more relations 
for a presentation whose proof may also be not similar to ours. 
\end{rem}

A consequence of Corollary~\ref{presentation} is a new presentation of Demazure operators, which is a `zoomed-out' version of the `zoomed-in' presentation in \cite{KELP1}. 
To make this precise, we introduce a subcategory.

\begin{defn}\label{def.demcore}
The \emph{atomic nilCoxeter algebroid}\footnote{Denoted by $\mathcal AD$ in \cite{KELP3}; see Proposition \ref{Dcore=AD} and \cite[Definition 2.25, Theorem 2.26]{KELP3}} 
$\Dem^{\core}$
is the subcategory of the nilCoxeter algebroid $\Dem$ 
spanned by the morphisms $\partial_\pcoset$ associated to core cosets $\pcoset$.   
\end{defn}

Note that $\{\pa_p\ |\ p\in\SC\text{ is a core coset}\}$ is a basis of $\Dem^{\core}$ since $\pa_p$, for all double cosets $p$, form a basis of $\Dem$ (see \cite[Corollary 3.20]{KELP1}).

\begin{prop}\label{Dcore=AD}
Definition~\ref{def.demcore} does define a category, that is, $\Dem^{\core}$ is closed under composition. 
In fact, $\Dem^{\core}$ is the subcategory of $\Dem$ generated by the atomic cosets.
\end{prop}
\begin{proof}
\rev{By Lemma~\ref{corecore},}
a reduced composition of core cosets \rev{belongs to $\Dem^{\core}$.}
But any non-reduced composition in $\Dem$, being zero, also belongs to $\Dem^{\core}$. 
This proves the first statement. 
The last statement follows from Proposition~\ref{prop.arexfinns} (see also \cite[Definition 2.25, Theorem 2.26]{KELP3}).
\end{proof}


For concreteness, for $I\subset S$ finitary and $s\in S\setminus I$, we denote by $\!^I_s\at$ (resp., $\at^I_s$) the (unique) atomic coset of the form $\!^I_s\at= W_Iw_{Is}W_J$ (resp., $\at^I_s= W_Jw_{Is}W_I$). 
Thus, for $Is=Jt$ and $s=w_{Jt}tw_{Jt}$ we have 
\[\!^I_s\at\expr [I,Is,J] = [I,Jt,J ]\expr \at^J_t\]
(see Remark~\ref{rem.I,s=atoms}). 
This provides a purely combinatorial indexing of the atomic cosets. Note that 
\begin{equation}\label{eq.ainv}
    \!^I_s\at= (\at^I_s)\inv
\end{equation} (see Remark~\ref{rem.inverseatom}).

\begin{thm}\label{thm.presentDemazure}
The category $\Dem^{\core}$, associated to a Coxeter system $(W,S)$, admits the following presentation by generators and relations, as a $\kk$-linear category. \rev{It is generated by }
\[\{ \pa_{\at^J_t}\ |\ J\subset S,\ t\in S\setminus J,\ Jt \text{ finitary}\}=\{\pa_\at\ |\ \at\text{ is an atom}\}=\{ \pa_{^I_s\!\at}\ |\ I\subset S,\ s\in S\setminus I,\ Is \text{ finitary}\}\]
\rev{subject to the relations}
\begin{itemize}
    \item for each finitary subset $I\subset S$ and elements $s\neq t$ in $S\setminus I$, the atomic braid relation \[
    \underbrace{\pa_\at\circ \pa_{\atb'}\circ \cdots}_{m} = \underbrace{\pa_\atb\circ \pa_{\at'}\circ \cdots}_{m},\]
    where the atoms $\at',\atb',\cdots$ and $m\in \mathbb Z_{\geq 2}$ depend on $\at=\!^I_s\at$ and $\atb=\!^I_t\at$ (see Definition~\ref{def.braid});
    \item for each atom $\at=\at^I_s$
    the \emph{atomic nil-quadratic relation}  
    \[\partial_\at\circ \partial_{\at\inv}=
    0,\qquad
    \text{ i.e.,}\qquad 
    \pa_{\at^I_s}\circ \pa_{^I_s\!\at}=0.\]
\end{itemize}
\end{thm}
\begin{proof}
\rev{Proposition~\ref{Dcore=AD} shows that the atomic Demazure oparators generate $\Dem^{\core}$.}
Since an atomic braid relation is a composition of singular braid relations, the atomic braid relations hold in $\Dem$. 
That the atomic nil-quadratic relations hold follows from \rev{that  
\[\at^I_s * \!^I_s\at \expr [Is\setminus t+t-s]\circ[I+s-t]=[Is\setminus t +t\underbrace{-s+s}-t],\]
 where $t=w_{Is}sw_{Is}$, is an expression containing $[Is-s+s]$.}

Now we show that the above \rev{list of relations is complete}.
Theorem~\ref{mats} covers the reduced expressions, i.e., the nonzero compositions in $\Dem^{\core}$. 
It remains to show that if $\partial_{[\at_\bullet]}=\partial_{\at_1}\circ \partial_{\at_2}\circ \cdots \circ \partial_{\at_m}$ is zero, for $\at_i$ atomic, then the expression $\at_\bullet=[\at_0,\cdots,\at_m]$ is related by atomic braid relations to an expression $[\atb_0,\cdots,\atb_m]$ with  $\atb_k=\atb_{k+1}\inv$ at some index $k$.
Since 
\rrev{the left hand side of the atomic quadratic relation \eqref{quadratic} is of the form $\at *\at\inv$,} Corollary~\ref{presentation} gives the claim for when $\at_\bullet$ is an admissible atomic expression. 


Now suppose $\at_\bullet$ is not an admissible atomic expression. In this case, there exists a proper admissible subexpression $\qcoset= \at_0*\at_1*\cdots * \at_k$ of atomic length $\geq 1$ such that $r = \at_0*\cdots * \at_k*\at_{k+1}$ is not a core coset.
If $\qcoset= \at_0*\at_1*\cdots * \at_k$ is not reduced then, as in the above paragraph,  Corollary~\ref{presentation} gives the claim.
If $\qcoset= \at_0*\at_1*\cdots * \at_k$ is reduced, then $\qcoset$ has an atomic reduced expression ending at $[L+u-s]$, where $s\in S$ is such that $\at_{k+1} \expr [K+s-t]$, by Lemma~\ref{lem.exchange}.
By Theorem~\ref{mats}, we may assume that $\at_k\rev{\expr}[L+u-s]$.
But $t=w_{Ks}sw_{Ks} = u$ implies $\at_k=\at_{k+1}\inv$. 
\end{proof}

\part{Core combinatorics}
\section{Corank 2 case}\label{s.dihedral}

Let $(W,S)$ be a finite Coxeter system.
Note that if $|S\setminus I| = |S\setminus J| \in\{0,1\}$ then a core $(I,J)$-coset is either the identity coset or the unique atomic $(I,J)$-coset and thus has a unique reduced expression which is atomic.
In this section we consider the next easiest case, namely, the core $(I,J)$-cosets where $|S\setminus I| = |S\setminus J| = 2$.
Our goal is to explain how the discussion from Section~\ref{s.switchbacks}  completely describes the atomic reduced expressions.

We denote the dihedral group of order $2m$ by $W(I_2(m))$ viewed as the Coxeter group with the generators $\{a,b\}$ and the relation $\underbrace{ab\cdots}_m =\underbrace{ba\cdots}_m$. 
Then $W(I_2(m))$ has exactly $2m+1$ reduced expressions, whose list is
\begin{equation}\label{eq.rexsdihed}
e,\ a,\ b,\ ba,\ ab,\ \cdots\ ,\ \ka{m},\ \kb{m}.    
\end{equation}
where we use the notation 
\[ \ka{k} = \underbrace{\cdots ba}_k ,\quad \kb{k}= \underbrace{\cdots ab}_k.\]

The following lemma is implicit in the construction of the switchback relation.
\begin{lem}\label{lem.ka}
Let $J\subset S$ be such that $S\setminus J = \{s_1,s_2\}$ for $s_1\neq s_2$ and let $k>0$. 
For $i=1,2$, there exists exactly one atomic expression, denoted by $E(i,k)$, ending in $[Js_i,J]$ of atomic length $k$.
The expression $E(1,k)$ is reduced if and only if $k\leq d+1$, where  $d+1$ is the atomic length of the atomic braid relation associated to $(S,\textcolor{RoyalBlue}{\con{s_1}},\textcolor{Green}{s_2})$, 
i.e., $d=d(S,\textcolor{RoyalBlue}{\con{s_1}},\textcolor{Green}{s_2})$ 
(see Definition~\ref{def.braid} 
and recall $\textcolor{RoyalBlue}{\con{s_1}} = w_S \textcolor{Orange}{s_1} w_S$).
Moreover, $E(1,d+1)\expr E(2,d+1)$ is an atomic braid relation.
\end{lem}
\begin{proof}
Note that there exist exactly two atomic $(I,J)$-cosets, for exactly two $I\subset S$, namely, $[Js_i\setminus w_{Js_i}s_iw_{Js_i},Js_i,J]$. 
Similarly, if $L\subset S$ is any subset with $|L|=|J|$, then there exists exactly two atomic $(I,L)$-cosets where $I$ \rev{varies}.
Thus, at each step of composing atomic cosets, from the right to left, we have exactly two choices of composable atomic cosets. 
Except at the first step, one of the two choices yields a non-reduced subexpression of the form $[+s-t+t-s]$.

It follows that there is a unique atomic expression, call it $E(i,k)$, of atomic length $k>0$ ending at $[Js_i,J]$, which does not contain a subexpression of the form $[+s-t+t-s]$.
By Proposition~\ref{prop.atominrotation}, this is of the form \eqref{eq.tok2} for the rotation sequence $u_\bullet$ associated to $(S,\con{u_{k}},u_{k-1})$(these are the same rotation sequence up to shift, which arises since we start from the right here).
If $k=d+1$ then we obtain the atomic expression in the left hand side of \eqref{eq.atombraid}. 
The claim thus follows from Definition~\ref{def.d} and Proposition~\ref{prop.braid}.
\end{proof}

\begin{prop}\label{prop.dihed}
Let $(W,S)$ be a finite Coxeter system and let  $J\subset S$ with $|S\setminus J|=2$.
Then we have a bijection
    \[\phi:W(I_2(d+1)) \to \{\text{core $(I,J)$-cosets where $I\subset S$ varies}\}=: \coreJ{J}\]
where $d+1$ is the atomic length of the atomic braid relation associated to  $(S,S\setminus J)$ (see Definition~\ref{def.braid} and remark~\ref{rem.param}).
Moreover, the bijection is compatible with the braid relations in the sense that the only braid relation 
\[\ka{m} \expr \kb{m}\]
in $W(I_2(d+1))$ 
corresponds to the only atomic braid relation in $\coreJ{J}$, which is on the element $\phi(\ka{m})=\phi(\kb{m})$.
\end{prop}
\begin{proof}
We construct two such bijections $\phi_1, \phi_2$ in the proof. Let $S\setminus J = \{s_1,s_2\}$.  Then Lemma~\ref{lem.ka} and the related discussion in Section~\ref{s.switchbacks} applies. 

Define 
\begin{equation}
\Phi_i: \{\text{all reduced expressions in $I_2(d+1)$}\} \to \{\text{atomic reduced expressions for $(W,S)$}\}     
\end{equation}
by letting $\Phi_1(\ka{k})=E(1,k)$, $\Phi_1(\kb{k})=E(2,k)$, from Lemma~\ref{lem.ka} and the other way for $\Phi_2$.  
Each $\Phi_i$ is injective by construction.

The only braid relation in the domain of $\Phi_i$ is $\ka{d+1}\expr \kb{d+1}$. Lemma~\ref{lem.ka} says that $\Phi_i$ maps this to the atomic braid relation $\Phi_i(\ka{d+1})\expr \Phi_i(\kb{d+1})$. By Proposition~\ref{prop.suds}, the latter is the only atomic braid relation in the image of $\Phi_i$.

Now the classical Matsumoto theorem and the atomic Matsumoto theorem (Theorem~\ref{mats}) together show that $\Phi_i$ induces a well-defined injection $\phi_i$ satisfying the second claim of the proposition. That $\phi_i$ is surjective follows from Proposition~\ref{prop.arexfinns}.
This completes the proof.
\end{proof}

The following examples are deduced from \cite[Section 6]{EKo} and Proposition~\ref{prop.dihed}.

\begin{ex}\label{Dbraid}
Let $(W,S)$ be of type $D$. Let $s_1\in S$ be the unique element having three edges in the associated Dynkin diagram. 
If $s_1\in J\subset S$ with $|S\setminus J|=2$, then   $\coreJ{J}$ and its atomic reduced expressions correspond to that of the dihedral group of type $A_2$.
If $s_1\not\in J\subset S$ with $|S\setminus J|=2$, then $\coreJ{J}$ and its atomic reduced expressions correspond to that of the dihedral group of type $B_2$.
\end{ex}


\begin{ex}\label{H4braid}
Let $(W,S)$ be of type $H_4$ with the indexing 
\begin{tikzpicture}[scale=0.4,baseline=-3]
\protect\draw (4 cm,0.05cm) -- (2 cm,0.05 cm);
\protect\draw (4 cm,-0.05cm) -- (2 cm,-0.05 cm);
\protect\draw (4 cm,0cm) -- (2 cm,0 cm) ;
\protect\draw (2 cm,0) -- (0 cm,0);
\protect\draw (0 cm,0) -- (-2 cm,0);
\protect\draw[fill=white] (3 cm, 0 cm) circle (0cm) node[below=0pt]{\tiny $5$};
\protect\draw[fill=white] (4 cm, 0 cm) circle (.15cm) node[above=1pt]{\scriptsize $4$};
\protect\draw[fill=white] (2 cm, 0 cm) circle (.15cm) node[above=1pt]{\scriptsize $3$};
\protect\draw[fill=white] (0 cm, 0 cm) circle (.15cm) node[above=1pt]{\scriptsize $2$};
\protect\draw[fill=white] (-2 cm, 0 cm) circle (.15cm) node[above=1pt]{\scriptsize $1$};
\end{tikzpicture}.
Then $\coreJ{\{s_3,s_4\}}$ and its atomic reduced expressions correspond to that of the dihedral group of type $I_2(10)$.
If $J \subset S$ is any subset of order two other than $\{s_3,s_4\}$, then $\coreJ{J}$ and its atomic reduced expressions correspond to that of the dihedral group of type $I_2(12)$.
\end{ex}

\begin{rem}
If $J\subset S$ is such that $S\setminus J=\{s,t\}$ with $w_{Js}sw_{Js} = s$ and $w_{Jt}tw_{Jt} = t$ then $\SC^{\core}_J$ consists only of $(J,J)$-cosets.
Thus $\SC^{\core}_J$
is a monoid and $\phi_i$ are upgraded to isomorphisms from $(W(I_2(d+1)),*)$ (viewed as a category with one object) to $\SC^{\core}_J$ so that $\phi_1\inv\circ \phi_2$ is the automorphism on $(W(I_2(d+1)),*)$ swapping the two Coxeter generators. Similarly, we have an algebra isomorphism between the nilCoxeter algebra of type $I_2(d+1)$ and $\Dem^{\core}_J$ in this case.

In general, one can define an equivalence relation on the subsets of $S$ generated by $K\sim L$ if there exists an atomic $(K,L)$-coset and consider the subcategory
\begin{equation}\label{eq.DemJeqclass}
\Dem^{\core}_{(J)}:=\bigcup_{K\in(J)}\Dem^{\core}_J 
\end{equation}
in $\Dem^{\core}$ associated to the equivalence class $(J)$ of $J$.
Then $\Dem^{\core}_{(J)}$ is a multiple object version of the nilCoxeter algebra $D(I_2(d+1))$.
Viewing a bijection $\phi^K$ from Proposition~\ref{prop.dihed}, for each $K\in(J)$, as a bijection from $D(I_2(d+1))$ to $\Dem^{\core}_K$, we have a decomposition
\[\Dem^{\core}_{(J)} = \bigsqcup_{K\in (J)} \phi^K(D(I_2(d+1))).\] 
\end{rem}

In our discussion in this section, the ambient Coxeter system $(W,S)$ is assumed to be finite in order for Section~\ref{s.switchbacks} to apply. The next remark  provides some complement to this.

\begin{rem}\label{rem.infiniterank2}
Let $(W,S)$ be infinite and $J\subset S$ finitary with $S\setminus J = \{s_1,s_2\}$ for $s_1\neq s_2$.
If $Js_i\subset S$ is not finitary then there is no $(I,J)$-expression ending in $[-s_i]$, for any $I\subset S$. 
If $Js_1\subset S$ is finitary, then a reduced $(I,J)$-expression ending in $[-s_1]$ is still of the form $E(1,k)$, which exists for $k< m$ where an infinitary subset $K\subset S$ is first involved at the $m$-th step. In particular, if $S$ is of affine type then we have well-defined and distinct $E(1,k)$ and $E(2,k)$ for all $k$ which identifies $\SC^{\core}_J$ with the Coxeter monoid of type $I_2(\infty)$.
\end{rem}


\section{General case}\label{s.highrank}

\subsection{The poset of left/right-fixed core cosets}
Let $(W,S)$ be a Coxeter system.
The rest of the paper is devoted to general and special discussions on
\begin{equation}\label{eq.defcoreJ}
\coreJ{J}:=\{\text{core $(I,J)$-cosets where $I\subset S$ varies}\} \end{equation}
for $J\subset S$ finitary.

Since non-reduced expressions do not appear in the rest of the paper, the discussions below are valid, up to linearization, for 
\[\Dem^{\core}_J=\bigcup_{I\in\Dem^{\core}}\Hom_{\Dem^{\core}}(I,J)\]
which is more naturally defined from the category $\Dem^{\core}$ (recall that the core cosets in $\SC$ do not form a subcategory).

We consider the left-fixed versions
\begin{equation*}
\!_{I}\SC^{\core}:=\{\text{core $(I,J)$-cosets where $J\subset S$ varies}\},\quad \!_{I}\Dem^{\core}=\bigcup_{J\in\Dem^{\core}}\Hom_{\Dem^{\core}}(I,J) 
\end{equation*}
as well.

In the special cases covered in Section~\ref{s.dihedral} or in \cite{KELP3}, the atomic reduced expressions and the atomic braid relations in $\coreJ{J}$
agrees with the ordinary reduced expressions and the ordinary braid relations in some Coxeter system. This is not the case in general. The easiest such example is the following.

\begin{ex}\label{ex.d}
The order of $\coreJ{\{s\}}$ in type $D_4$, for any $s\in S$, is $32$. Moreover, each such $\coreJ{\{s\}}$ has is a unique maximal element of atomic length $7$. 
(While these numbers are the same, the structure of $\coreJ{\{s\}}$, e.g., the atomic reduced expression graphs, does depend on $s\in S$. 
We discuss one of the two, up to symmetry, cases in detail in Section~\ref{ss.trihedral}.) 
\end{ex}

Here is some more result of computations on SageMath.
\begin{ex}\label{ex.H4numbers}
We consider the four corank 3 cases in type $H_4$. We index $S=\{s_1,s_2,s_3,s_4\}$ as in \begin{tikzpicture}[scale=0.4,baseline=-3]
\protect\draw (4 cm,0.05cm) -- (2 cm,0.05 cm);
\protect\draw (4 cm,-0.05cm) -- (2 cm,-0.05 cm);
\protect\draw (4 cm,0cm) -- (2 cm,0 cm) ;
\protect\draw (2 cm,0) -- (0 cm,0);
\protect\draw (0 cm,0) -- (-2 cm,0);
\protect\draw[fill=white] (3 cm, 0 cm) circle (0cm) node[below=0pt]{\tiny $5$};
\protect\draw[fill=white] (4 cm, 0 cm) circle (.15cm) node[above=1pt]{\scriptsize $4$};
\protect\draw[fill=white] (2 cm, 0 cm) circle (.15cm) node[above=1pt]{\scriptsize $3$};
\protect\draw[fill=white] (0 cm, 0 cm) circle (.15cm) node[above=1pt]{\scriptsize $2$};
\protect\draw[fill=white] (-2 cm, 0 cm) circle (.15cm) node[above=1pt]{\scriptsize $1$};
\end{tikzpicture}.
In all cases, the set $\coreJ{\{s_i\}}$ is of order 480 and contains a unique maximal element of atomic length 31. 
When $i=1$ the number of elements of atomic length $j$ is the $j$-th number in the sequence
\[1, 3, 5, 7, 9, 11, 13, 15, 17, 19, 21, 23, 24, 24, 24, 24, 24, 24, 24, 24, 23, 21, 19, 17, 15, 13, 11, 9, 7, 5, 3, 1\]
and the longest element has 25392 atomic reduced expressions;
when $i=2$ we have the numbers
\[1, 3, 5, 7, 9, 12, 14, 15, 17, 19, 22, 23, 22, 22, 23, 26, 26, 23, 22, 22, 23, 22, 19, 17, 15, 14, 12, 9, 7, 5, 3, 1\]
and the longest element has 35032 atomic reduced expressions;
when $i=3$ we have
\[1, 3, 5, 7, 9, 12, 15, 16, 17, 19, 21, 22, 22, 23, 24, 24, 24, 24, 23, 22, 22, 21, 19, 17, 16, 15, 12, 9, 7, 5, 3, 1\]
and the longest element has 36874 atomic reduced expressions;
when $i=4$ we have
\[1, 3, 5, 7, 9, 12, 15, 17, 18, 18, 19, 21, 23, 24, 24, 24, 24, 24, 24, 23, 21, 19, 18, 18, 17, 15, 12, 9, 7, 5, 3, 1\]
and the longest element 36746 atomic reduced expressions.
That all atomic length sequences are symmetric is because the involution $s\mapsto\con{s}$ is the identity in type $H$, which we explain in Proposition~\ref{prop.w0-}.
\end{ex}

The objects $\coreJ{J}$ do share a number of general properties of the Coxeter monoids. We investigate some of such in this section.
Recall that, for $S$ finitary, we set $\con J = w_SJw_S$ where $w_S$ is the longest element in $(W_S,S)$.

\begin{lem}\label{lem.maxcore}
Let $S$ be finitary.
Then for $I,J\subset S$, the $(I,J)$-coset $W_Iw_SW_J$ is core if and only if $I=\con J$.
\end{lem}
\begin{proof}
The claim is trivial when $|I|\neq |J|$, so we assume $|I|=|J|$.

Let $\pcoset= W_{I} w_S W_J$ and consider the element $w_Sw_J\inv\in \pcoset$.
We write $w_Sw_J\inv = x.\mi{p}$ for some $x\in W_I$.
Then 
\begin{equation}\label{pJp}
\mi{p} J \mi{p}\inv = x\inv w_S w_J J w_J w_S x = x\inv w_S J w_S x = x\inv \con{J} x     
\end{equation}
holds.
Now we have
\[\leftred(p) \underset{\text{def.}}{=} I\cap \mi{p}J\mi{p}\inv \underset{}{=} W_I\cap\mi{p}J\mi{p}\inv \underset{\eqref{pJp}}{=} W_I\cap x\inv \con{J} x \underset{\text{conjugate}}{\cong} x W_I x\inv \cap \con{J} \underset{x\in W_I}{=} W_I\cap \con{J} = I\cap \con{J}\]
where the second equality is due to Kilmoyer's theorem \cite[Lemma 2.14]{EKo}.
Since $\pcoset$ is core if and only if $\leftred(p) = |J|$, this proves the claim.
\end{proof}

An immediate consequence of Lemma~\ref{lem.maxcore} is:
\begin{cor}\label{cor.uniquemax}
Let $S$ be finitary and $J\subset S$. The core coset $\pcoset_J=W_{\con J} w_S W_J$ is the unique maximal element in $\coreJ{J}$.
\end{cor}

In a Coxeter group, multiplication by the longest element $w_0$ induces an anti-involution on the strong and weak Bruhat poset. 
The $*$-multiplication by an element cannot induce an anti-involution, but since a Coxeter monoid agrees with the Coxeter group as a set, we have the same anti-involution on a Coxeter monoid. 
To define such anti-involutions directly in terms of the $*$-multiplication, we do as follows. 
Given $w\in W$, there is a unique element $x\in W$ satisfying $w.x=w_0$. Letting $w\mapsto x$, we obtain an anti-involution on the left Bruhat order. 
Proposition~\ref{prop.w0-} is its atomic analogue.

We first fix notation for the left and right Bruhat orders on $\SC$. The (weak) left Bruhat order is defined and denoted as
\[p\leq_l q \text{ if $\qcoset = r.\pcoset$ for some $r\in \SC$ ;}\]
the (weak) right Bruhat order is defined and denoted as
\[p\leq_r q \text{ if $\qcoset = p.r$ for some $r\in \SC$} .\]

\begin{prop}\label{prop.w0-}
Let $S$ be finitary and $J\subset S$.
If $\con{J} = J$ then there exists an anti-involution of the poset $(\coreJ{J},\leq_l)$.
More generally, there is an anti-isomorphism of the posets $(\coreJ{J},\leq_l)\to (_{\con J}\SC^{\core},\leq_r)$.
\end{prop}
\begin{proof}
Given a core $(I,J)$-coset $\pcoset$,
by (the right-to-left analogue of) \cite[Proposition 4.4]{EKo} and its proof, there is a unique $(\con J, I)$-coset $\qcoset$ such that $\qcoset*\pcoset = W_{\con J} w_S W_J$.
The assignment $\pcoset\mapsto\qcoset$ defines a map 
\[f:(\coreJ{J},\leq_l)\to ( _{\con J}\SC^{\core},\leq_r)\] 
whose inverse is given by a similar map. Thus $f$ is the anti-isomorphism claimed in the more general statement.

Now, if $J=\con J$ then reversing expressions induces a poset isomorphism  $g:( _{ J}\SC^{\core},\leq_r)\to (\coreJ{J},\leq_l)$. 
The anti-involution $g\circ f$ proves the first claim.
\end{proof}

One may also consider the strong Bruhat order on $\coreJ{J}$ as restricted from the Bruhar order on $\SC$ (see \cite{KELP2}), but for the purpose of studying $\coreJ{J}$ the weak Bruhat order seems more natural.


\subsection{A corank 3 example}\label{ss.trihedral}

Let $(W,S)$ be of type $D_4$ and let $c\in S:=\{c,t,u,v\}$ be the simple generator corresponding to the center of the Dynkin diagram of type $D_4$. In this subsection we describe the atomic reduced expressions for $P:=\coreJ{\{c\}}$.
The atomic cosets in $P$ are 
\[\at^c_t :\expr [t, ct,c],\quad \at^c_u :\expr [u,cu,c],\quad \at^c_v :\expr [v,cv,c].\]
Here $t$ is a shorthand for $\{t\}$, $ct$ denotes $\{c,t\}$, and so on.
Note that the $S_3$-symmetry permuting $t,u,v$ induces an $S_3$-symmetry on $P$. 
The other atomic cosets that appears in atomic reduced expressions in $P$ are
\[\at^t_c:\expr [c,ct,t],\quad \at^t_u:\expr [t,tu,t]\]
up to the $S_3$-symmetry. We drop the superscripts when it is clear from the context: for example,
there is a unique $s\in S$ which makes $\at^s_v*\at^c_t$ well-defined, namely $s=t$, so we write $\at_v*\at^c_t$ instead. 
Since each element in $P$ can be composed on the left by three atoms, and one of those three produces a subexpression of the form $[-s]\circ [+s]$, the number of reduced expressions in $P$ of (atomic) length $m$ is bounded by $3\cdot 2^{m-1}$.

A direct computation (or \cite[Section 6.4]{EKo}) determines the full list of atomic braid relations in $P$. 
The atomic braid relations for $(I,\{c\})$-cosets are 
\begin{equation}\label{withc}
    \at_c\at_t\at^c_u = \at_c\at_u\at^c_t,\quad \at_c\at_t\at^c_v = \at_c\at_v\at^c_t,\quad \at_c\at_v\at^c_u = \at_c\at_u\at^c_v,
\end{equation}
but we also need the atomic braid relations
\begin{equation}\label{noc}
    \at_t\at_u^v
    = \at_u\at_t^v
    , \quad \at_t\at_v^u
    = \at_v\at_t^u
    , \quad \at_v\at_u^t
    = \at_u\at_v^t,
\end{equation}
since they give relations in $P$, for example,
\begin{equation}\label{noc3}
    \at_t\at_u^v\at^c_v 
    = \at_u\at_t^v\at^c_v
    , \quad \at_t\at_v^u
    \at^c_u 
    = \at_v\at_t^u\at^c_u
    , \quad \at_v\at_u^t\at^c_t 
    = \at_u\at_v^t\at^c_t
    .
\end{equation}

Then it is easy to see (or use Corollary~\ref{presentation}) that there are 6 elements in $P$ of (atomic) length two, each of which has a unique reduced expression.
From \eqref{noc3} and \eqref{withc}, it follows that the 12 potentially-reduced length three expressions are indeed reduced and give 6 elements in $P$ of length three. Each of these has two reduced expressions.

There is only one way to extend a length three expression in $P$ to a length four expression in $P$, because a length three expression is related by braid relation to another that has a different atom on the left. 
So we have 6 elements of length four in $P$ each of which has two reduced expressions.

Now we may stop computing and apply Proposition~\ref{prop.w0-}, since the maximal element 
\[W_{c}w_S W_c= \{w_S, cw_S=w_Sc\}\in P\] has atomic length seven. This determines in particular the number of elements in $P$ of length four, five, six, seven, namely, 6, 6, 3, 1, respectively, as well as their weak Bruhat relations. 
This also determines the number of reduced expressions for each element, namely all elements of length four (resp., five, six, seven) has 2 (resp., 4, 8, 24) reduced expressions.



A few additional direct computations determine the Hasse diagram for the weak Bruhat poset $P$, which we include below. The diagram is aligned according to the atomic lengths. The number on each vertex is the number of atomic reduced expressions for the represented element.

\[
\begin{tikzpicture}[scale=0.8]
  \node (1) at (0,2) {$1$};
  \node (11) at (-2,1) {$1$};
  \node (12) at (0,1) {$1$};
  \node (13) at (2,1) {$1$};
  \node (21) at (-5,0) {$1$};
  \node (22) at (-3,0) {$1$};
  \node (23) at (-1,0) {1};
    \node (24) at (1,0) {$1$};
  \node (25) at (3,0) {$1$};
  \node (26) at (5,0) {1};
  \node (31) at (-5,-1) {$2$};
  \node (32) at (-3,-1) {$2$};
  \node (33) at (-1,-1) {2};
    \node (34) at (1,-1) {$2$};
  \node (35) at (3,-1) {$2$};
  \node (36) at (5,-1) {2};
  \node (41) at (-5,-2) {$2$};
  \node (42) at (-3,-2) {$2$};
  \node (43) at (-1,-2) {2};
    \node (44) at (1,-2) {$2$};
  \node (45) at (3,-2) {$2$};
  \node (46) at (5,-2) {2};
  \node (51) at (-5,-3) {$4$};
  \node (52) at (-3,-3) {$4$};
  \node (53) at (-1,-3) {$4$};
  \node (54) at (1,-3) {$4$};
  \node (55) at (3,-3) {$4$};
  \node (56) at (5,-3) {$4$};
  \node (61) at (-2,-4) {$8$};
  \node (62) at (0,-4) {$8$};
  \node (63) at (2,-4) {$8$};
  \node (7) at (0,-5) {$24$};
  \draw (1) -- 
  (11) -- (21) -- (31) -- (41) -- (51) --(61) -- (7)
  (11) -- (22) -- (32) -- (42) -- (51)
  (1) -- (12) -- (23) -- (31) 
  (43) -- (53) -- (62) -- (7)
  (12) -- (24) -- (34) -- (44) -- (52) -- (61)
  (1) -- (13) -- (25) -- (33) -- (43) -- (53) -- (62)
  (13) -- (26)
  (21) -- (32) 
  (22) -- (33)
  (23) -- (34)
  (24) -- (35)
  (25) -- (36)
  (26) -- (35)
  (26) -- (36)
  (35) -- (45) -- (55) -- (63) -- (7)
  (36) -- (46)
  (41) -- (52)
  (42) -- (53)
  (43) -- (54)
  (44) -- (55)
  (45) -- (56)
  (46) -- (54)
  (46) -- (56)
  (54) -- (62)
  (56) -- (63);
\end{tikzpicture}
\]

\begin{rem}
Proposition~\ref{prop.IW} below says that \cite[Example 3.12]{IW} illustrates our $P$. 
\end{rem}

\begin{rem}
If $S$ is of type $D_{n+2}$ and $S\setminus J =\{s_0,s_{\bar{0}}, s_n\}$ corresponds to the three extreme vertices in the Dynkin diagram, then the reduced expression graphs and the left Bruhat order on $\coreJ{J}$ are the same as our example in type $D_4$. The same explanation applies.
\end{rem}

\section{Tits cone intersections}\label{s.IW}

Let us consider the \emph{Tits cone} $\cone(W,S)$ associated to $(W,S)$ (see e.g. \cite[Section 1.3]{IW}, \cite[Section 2.6]{AbBr}, or \cite[p.123]{BjornerBrenti}). To briefly recall, the Tits cone is defined as the subset
\[\cone(W,S)=\bigcup_{w\in W}w(\overline{C})\subset V^*\] where  $V$ is the geometric (standard) realization of $(W,S)$ over $\mathbb R$ with the basis $\{\alpha_s\}_{s\in S}$ and  $\overline{C}$ is the closure of the standard chamber 
\[C=\{\theta\in V^*\ |\ \theta(\alpha_s)>0 \text{ for all }s\in S\}.\]

\begin{defn}\cite[Definition 1.5] {IW}\label{def.cone}
Given a subset $I\subset S$, the \emph{Tits $I$-cone} is the intersection 
\[\mathsf{Cone}(W,S,I) := \mathsf{Cone}(W,S)\cap \{ \theta \in V^* \ |\ \theta(\alpha_s)=0\text{ for }s\in I\}.\]
\end{defn}

The intersection $\cone(W,S,I)$ gives rise to a new kind of chamber geometry, as developed in \cite[Part I]{IW}. 
Let us fix $I\subset S$ for the rest of the section. 
For $J\subset S$, let
\[C_J:=\{\theta\in V^*\ |\ \theta(\alpha_s)>0\text{ for $s\in S\setminus J$ and $\theta(\alpha_s)=0$ for $s\in J$}\}\]
be the \emph{standard $J$-chamber} in $\cone(W,S,J)$. 
\rev{The image $w(C_J)$ under the action of any $w\in W$ is either contained in $\cone(W,S,I)$ 
or disjoint from $\cone(W,S,I)$. In fact, since the stabilizer of any $\theta\in w(C_J)$ is $wW_Jw\inv$, if $I\subset wW_Jw\inv$ then $w(C_J)\subset \cone(W,S,I)$, and if some $s\in I$ does not stabilize $w(C_J)$ then no element in $w(C_J)$ belongs to $\cone(W,S,I)$.}
A subset in $\cone(W,S,I)$ of the form $w(C_J)$, with some $w\in W$ and $J\subset S$ with $|J|=|I|$, is called an \emph{$I$-chamber}.

Let us denote by $\cham(W,S,I)$ the set of $I$-chambers in $\cone(W,S,I)$.
Then \cite[Theorem 1.12]{IW} provides the following relation between the Tits cone intersection and the core double cosets.
First note that, for a core $(I,J)$-coset $p$, we have $p=W_IxW_J=xW_J$ as sets, where $x\in p$. We thus let $p(C_J):=\mi{p}(C_J)=x(C_J)$.


\begin{prop}\label{prop.IW}
    For each finitary subset $I\subset S$, we have a bijection
    \begin{equation}
        \ _{I}\SC^{\core}(W,S) \rightarrow \mathsf{Cham}(W,S,I)
    \end{equation}
    sending a core $(I,J)$-coset $p$ to the chamber $p(C_J)$.
\end{prop}
\begin{proof}
For the claimed bijection to be well-defined, we need that $p(C_J)$ is a chamber in $\cone(W,S,I)$. But if $\theta\in C_J$ and  $s\in I$, then 
\[(p(\theta))(\alpha_s)=(\mi{p}(\theta))(\alpha_s)=\theta(\alpha_{\mi{p}\inv s\mi{p}}) =0 \] since 
\[\mi{p}\inv s\mi{p}\in \mi{p}\inv I\mi{p} =\mi{p}\inv \leftred(p)\mi{p} =\rightred(p)= J,\] which shows that $p(C_J)\subset \cone(W,S,I)$.

Now let $w(C_J)\in \cham(W,S,I)$, i.e., $w(C_J)\subset \cone(W,S,I)$ for some $w\in W$ and $J\subset S$ with $|J|=|I|$. 
We need to show that $W_Iw=wW_J$, from which it follows that $p:=W_IwW_J=W_Iw=wW_J$ is core 
and that our map is bijective.
But if $x\in W_I$ then the assumption 
 $w(C_J)\subset \cone(W,S,I)$ implies $xw(C_J) = w(C_J)$ and thus $xw\in wW_J$. This shows $W_Iw\subset wW_J$, which suffices because of the assumption $|I|=|J|$.
\end{proof}

Proposition~\ref{prop.IW} is merely an interpretation of \cite[Theorem 1.12]{IW} in our setting. Those readers familiar with \cite{IW} may prefer:
\begin{proof}[Proof of Proposition~\ref{prop.IW} which factors through \cite{IW}]
\cite[Theorem 1.12]{IW} states that 
$\cham(W,S,I)$ is in bijection with the set $\mathsf{cCham}(W,S,I)$ of \emph{combinatorial chambers} \footnote{Note that $\mathsf{cCham}(W,S,I)$ is what \cite{IW} denotes by $\cham(W,S,I)$. See \cite[Definition 1.8]{IW}.}, which consists of the pairs $(x, J)$, for $x\in W$ and $J\subset S$, satisfying
\begin{enumerate}[(i)]
    \item\label{cham1} $\ell(x) = \min\{\ell(y)\ |\ y\in xW_J\}$;
    \item\label{cham2} $W_Ix = xW_J$
\end{enumerate}
via $(x,J)\mapsto x(C_{J})$.
It is therefore enough to show that the map $\!_{I}\SC^{\core}(W,S) \rightarrow \mathsf{cCham}(W,S,I)$
given by $p\mapsto (\mi{p},J)$ is a bijection. 
In fact, 
\eqref{cham2} 
is equivalent to $W_IxW_J$ being a core $(I,J)$-coset, so  $(x,J)\mapsto W_IxW_J
$ defines a left inverse. 
 Condition~\eqref{cham1} 
makes it a right inverse.
\end{proof}

Proposition~\ref{prop.IW} readily establishes a similar connection to the atomic nilCoxeter algebroid as well.

\begin{prop}\label{prop.basis}
 We have
\[\bigsqcup_{I\subset S \text{ finitary}}\cham(W,S,I)\xhookrightarrow[\text{basis}]{} \Dem^{\core}(W,S)\]
which restricts to 
\[\bigsqcup_{K\in (I)}\cham(W,S,K)\xhookrightarrow[\text{basis}]{} \!_{(I)}\Dem^{\core}(W,S)\]
for each $I\subset S$ (see \eqref{eq.DemJeqclass}). 
\end{prop}
\begin{proof}
Since the categories $\SC$ and $\Dem$ have the same morphisms, up to linearization, and the same reduced expression with respect to the Coxeter presentations (see Section~\ref{s.prelim}), we have
\begin{equation}\label{eq.SC=SD}
    \ _{I}\SC^{\core}(W,S) \xhookrightarrow[\text{basis}]{} \ _{I}\Dem^{\core}(W,S).
\end{equation}
for each $I\subset S$ finitary.
The claim follows by Proposition~\ref{prop.IW}.
\end{proof}

Proposition~\ref{prop.basis} interprets Theorem~\ref{thm.presentDemazure} as a generators and relations presentation of all (resp., equivalent) $I$-chambers.



\begin{rem}
    While the current paper only considers finitary subsets $I\subset S$, there is no finiteness assumption in defining $\cone(W,S)$ and $\cham(W,S)$. However, for the application in \cite{IW} to tilting theory and birational geometry, it is assumed that $S$ is affine or finite. 
    In either case, a proper subset in $S$ is automatically finitary, so our setting suffices for such purposes. 
\end{rem}

\printbibliography

@misc{IW,
      title={Tits Cone Intersections and Applications}, 
      author={Iyama, Osamu and Wemyss, Michael},
      year={accessed on 20 Nov. 2023},
    note={preprint},
URL={https://www.maths.gla.ac.uk/~mwemyss/MainFile_for_web.pdf}
      }

@manual{sage,
  Key          = {SageMath},
  Author       = {{The Sage Developers}},
  Title        = {{S}ageMath, the {S}age {M}athematics {S}oftware {S}ystem ({V}ersion 9.5)},
  note         = {{\tt https://www.sagemath.org}},
  Year         = {2022},
}

@book{AbBr,
	author = {Abramenko, Peter and Brown, Kenneth S.},
	mrnumber = {2439729 (2009g:20055)},
	mrreviewer = {Ralf Gramlich},
	note = {Theory and applications},
	pages = {xxii+747},
	publisher = {Springer, New York},
	series = {Graduate Texts in Mathematics},
	title = {Buildings}, 
	volume = {248},
	year = {2008}}

@misc{KELP1,
      title={Demazure operators for double cosets}, 
      author={Elias, Ben and Ko, Hankyung and Libedinsky, Nicolas and Patimo, Leonardo},
      year={2023},
note={arXiv:2307.15021}
      }

@misc{KELP2,
      title={Subexpressions and the Bruhat order for double cosets}, 
      author={Elias, Ben and Ko, Hankyung and Libedinsky, Nicolas and Patimo, Leonardo},
      year={2023},
note={arXiv:2307.15726}
      }

@misc{KELP3,
      title={On reduced expressions for core double cosets}, 
      author={Ben Elias and Hankyung Ko and Nicolas Libedinsky and Leonardo Patimo},
      year={2024},
      eprint={2402.08673},
      archivePrefix={arXiv},
      primaryClass={math.CO}
}

@misc{KELP4,
      title={Singular Light Leaves}, 
      author={Ben Elias and Hankyung Ko and Nicolas Libedinsky and Leonardo Patimo},
      year={2024},
      eprint={2401.03053},
      archivePrefix={arXiv},
\\      primaryClass={math.RT}
}

@book {BjornerBrenti,
    AUTHOR = {Bj\"{o}rner, Anders and Brenti, Francesco},
     TITLE = {Combinatorics of {C}oxeter groups},
    SERIES = {Graduate Texts in Mathematics},
    VOLUME = {231},
 PUBLISHER = {Springer, New York},
      YEAR = {2005},
     PAGES = {xiv+363},
%      ISBN = {978-3540-442387; 3-540-44238-3},
   MRCLASS = {05-01 (05E15 20F55)},
  MRNUMBER = {2133266},
MRREVIEWER = {Jian-yi\ Shi},
}

@misc{Patimo,
      title={Charges via the {A}ffine {G}rassmannian}, 
      author={Patimo, Leonardo},
      year={2021},
      eprint={arXiv:2106.02564},
      archivePrefix={arXiv},
      primaryClass={math.CO},
      note={arXiv:2106.02564},
}

@article {SingSb,
    AUTHOR = {Williamson, Geordie},
     TITLE = {Singular {S}oergel bimodules},
   JOURNAL = {Int. Math. Res. Not. IMRN},
  FJOURNAL = {International Mathematics Research Notices. IMRN},
      YEAR = {2011},
    NUMBER = {20},
     PAGES = {4555--4632},
   MRCLASS = {20F55 (18D05)},
  MRNUMBER = {2844932},
MRREVIEWER = {G\"{o}tz Pfeiffer},
   
}

@article {EKo,
    AUTHOR = {Elias, Ben and Ko, Hankyung},
     TITLE = {A singular {C}oxeter presentation},
   JOURNAL = {Proc. Lond. Math. Soc. (3)},
  FJOURNAL = {Proceedings of the London Mathematical Society. Third Series},
    VOLUME = {126},
      YEAR = {2023},
    NUMBER = {3},
     PAGES = {923--996},
      ISSN = {0024-6115},
   MRCLASS = {20F55 (20C08)},
  MRNUMBER = {4563864},
}

@PhdThesis{Wthesis,
  author  = {Williamson, Geordie},
  title   = {Singular {S}oergel bimodules},
  school  = {University of Freiburg},
  year    = {2008},
  mrclass = {Thesis},
  url     = {https://freidok.uni-freiburg.de/data/5093},
}

@article {Demazure,
    AUTHOR = {Demazure, Michel},
     TITLE = {Invariants sym\'{e}triques entiers des groupes de {W}eyl et
              torsion},
   JOURNAL = {Invent. Math.},
  FJOURNAL = {Inventiones Mathematicae},
    VOLUME = {21},
      YEAR = {1973},
     PAGES = {287--301},
     % ISSN = {0020-9910},
   MRCLASS = {14M15 (20G05 20H15 32M10)},
  MRNUMBER = {342522},
MRREVIEWER = {S. I. Gel\cprime fand},
       %DOI = {10.1007/BF01418790},
      % URL = {https://doi-org.uchile.idm.oclc.org/10.1007/BF01418790},
}
\end{document}